\documentclass[12pt,reqno]{amsart}
\usepackage{amssymb}
\usepackage{geometry}
\usepackage[latin1]{inputenc}
\usepackage[italian,english]{babel}
\usepackage{amsmath, amsfonts, amsthm}
\usepackage{graphicx}
\usepackage{pgfplots}
\usepackage{paralist}
\usepackage[style=british]{csquotes}
\usepackage{mathtools, mathabx,bigints}
\usepackage{comment}
\usepackage{booktabs}
\usepackage{subfig}
\usepackage{array}
\usepackage{multirow}
\numberwithin{equation}{section}
\newcolumntype{P}[1]{>{\centering\arraybackslash}p{#1}}

\newtheorem{thm}{Theorem}[section]
\newtheorem{lem}[thm]{Lemma}

\newtheorem{prop}[thm]{Proposition}

\newtheorem{defn}[thm]{Definition}
\theoremstyle{definition}
\newtheorem{rem}[thm]{Remark}
\theoremstyle{remark}

\theoremstyle{definition}
\newtheorem*{nota}{Notation}

\newcommand{\abs}[1]{\left\vert#1\right\vert}

\newcommand{\argmin}{\arg\min}

\newcommand{\OM}{\mathcal{M}}
\newcommand{\OMtilde}{\widetilde{\mathcal{M}}}

{\left\{\begin{array}{@{}l@{}}}{\end{array}\right.}
\patchcmd{\abstract}{\scshape\abstractname}{\textbf{\abstractname}}{}{}
\makeatletter 
\def\@makefnmark{} 
\makeatother 

\mathchardef\mhyphen="2D

\begin{document}
\title{Imaging of nonlinear materials via the Monotonicity Principle}
\author[V.Mottola, A. Corbo Esposito, G. Piscitelli, A. Tamburrino]{
Vincenzo Mottola$^1$, Antonio Corbo Esposito$^1$, Gianpaolo Piscitelli$^2$, Antonello Tamburrino$^{1,3}$}\footnote{\\$^1$Dipartimento di Ingegneria Elettrica e dell'Informazione \lq\lq M. Scarano\rq\rq, Universit\`a degli Studi di Cassino e del Lazio Meridionale, Via G. Di Biasio n. 43, 03043 Cassino (FR), Italy.\\
$^2$ Dipartimento di Scienze Economiche, Giuridiche, Informatiche e Motorie, Universit\'a degli Studi di Napoli Parthenope, 80035 Nola (NA), Italy. \\
$^3$Department of Electrical and Computer Engineering, Michigan State University, East Lansing, MI-48824, USA.\\
Email: vincenzo.mottola@unicas.it {\it (corresponding author)}, antonio.corboesposito@unicas.it,  gianpaolo.piscitelli@uniparthenope.it, antonello.tamburrino@unicas.it.}
\maketitle
\markright{IMAGING OF NONLINEAR MATERIALS VIA MONOTONICITY PRINCIPLE}

\begin{abstract} Inverse problems, which are related to Maxwell's equations, in the presence of nonlinear materials is a quite new topic in the literature. The lack of contributions in this area can be ascribed to the significant challenges that such problems pose. Retrieving the spatial behaviour of some unknown physical property, from boundary measurements, is a nonlinear and highly ill-posed problem even in the presence of linear materials. Furthermore, this complexity grows exponentially \textcolor{black}{in the presence of nonlinear materials}.

\textcolor{black}{In the tomography of linear materials, the Monotonicity Principle (MP) is the foundation of a class of non-iterative algorithms able to guarantee excellent performances and compatibility with real-time applications. Recently, the MP} has been extended to nonlinear materials under very general assumptions. Starting from the theoretical background for this extension, we develop a first real-time inversion method for the inverse obstacle problem in the presence of nonlinear materials. 

\textcolor{black}{The proposed method is intendend for all problems governed by the quasilinear Laplace equation, i.e. static problems involving nonlinear materials.}

In this paper, we provide some preliminary results which give the foundation of our method and some extended numerical examples. 


\end{abstract}

\section{Introduction} 
\label{sec1}
\textcolor{black}{This paper, proposes a real-time inversion method for the nonlinear Calder\'on problem~\cite{calderon1980inverse}. Specifically, the aim is to retrieve the spatial behaviour of the unknown coefficient $\gamma$ appearing in the following quasilinear elliptic partial differential equation}
\begin{equation}\label{eqn:dirint}
\textcolor{black}{
    \begin{cases}
        \nabla\cdot\left(\gamma(x,\abs{\nabla u})\nabla u\right)=0 & \text{in $\Omega$} \\
        u=f(x) & \text{on $\partial\Omega$}.
    \end{cases}
    }
\end{equation}
\textcolor{black}{In problem \eqref{eqn:dirint}, the equations are in the weak form, and $\Omega \subset \mathbb{R}^n$ ($n \ge 2$), the region under tomographic inspection, is a an open bounded connected domain with a Lipschitz boundary. The boundary data $f$ belongs to a suitable trace space $X_{\diamond}(\partial\Omega)$, defined in the following. The existence and uniqueness of the solution follow from assumptions discussed in Section~\ref{sec2}.}

\textcolor{black}{More precisely, the aim is to retrieve the spatial behaviour of the (nonlinear) unknown coefficient $\gamma_A$, starting from knowledge of proper boundary data, i.e. starting from knowledge of the Dirichlet-to-Neumann (DtN) operator}
\begin{equation*}
    \textcolor{black}{
    \Lambda:f\in X_{\diamond}(\partial\Omega)\to \gamma\partial_n u|_{\partial\Omega}\in X'_{\diamond}(\partial\Omega).
    }
\end{equation*}
\textcolor{black}{The targeted problem is the inverse obstacle problem, where the goal is to reconstruct the shape, position and dimension of one or more anomalies embedded in a known background and occupying region $A$. In other words, the function $\gamma$ specializes as $\gamma_A$, defined as}
\begin{equation}\label{eqn:mua}
\textcolor{black}{
    \gamma_A(x,s)=\begin{cases}
        \gamma_{nl}(x,s) & \text{in $A$}, \\
        \gamma_{bg}(x) & \text{in $\Omega\setminus A$},
    \end{cases}}
\end{equation}
\textcolor{black}{
where $\gamma_{nl}$ is the nonlinear material property, while $\gamma_{bg}$ is the linear material property. Both of which can be spatially dependent. We are interested in determining the region $A$ in which the coefficient $\gamma_A$ actually depends on $s=\lvert\nabla u\rvert$ and not only on the spatial coordinates.}

\textcolor{black}{Problem \eqref{eqn:dirint} is of particular interest since it is the model for various electromagnetic problems in steady-state condition, such as the magnetostatic case in the presence of} nonlinear magnetic materials, widely used in applications. One of the most significant examples regards electrical machines, where magnetic materials like electrical steel or permanent magnets play a paramount role. In the framework of inverse problems, there is a widespread demand for non-destructive, non-ionising methods able to detect a variety of materials, for surveillance and security reasons. One of the leading applications is the detection of magnetic materials in boxes or containers~\cite{art:Mar15, book:Dorn18}, 
\textcolor{black}{ but there is also a great interest in the inspection of concrete.} For example, reinforcing bars in concrete are typically made of steel, which may be subject to corrosion. Tomographic inspections can give useful information on the state of the material and, in particular, on the number, position and shape of the rebar inside the concrete~\cite{art:So05, art:Ig03}.

As well as magnetic materials, \textcolor{black}{problem~\eqref{eqn:dirint} is also a model for the steady currents problem involving nonlinear conductive materials. Nowadays, in addition to superconductors \cite{art:Super19}, materials exhibiting a nonlinear electrical conductivity are widely employed in field grading applications \cite{art:Bu08,art:Me18}. Furthermore, human tissues may exhibit a nonlinear electrical conductivity (\cite{art:tisnl,art:skin_nl}).} 

\textcolor{black}{Electrostatic phenomena involving nonlinear dielectrics are also modelled by Equation~\eqref{eqn:dirint}. In this field, ferroelectric materials play a key role~\cite{art:die_nl} in manufactoring tunable capacitors. Nonlinear dielectric materials have also been used in semiconductor structures such as Schottky junctions~\cite{art:diode_nl}.}

From a general perspective, the inverse problem in the presence of nonlinear materials is a quite new topic in the literature. As quoted in~\cite{art:Lam20} (2020), \enquote{\emph{the mathematical analysis for inverse problems governed by nonlinear Maxwell equations is still in the early stages of development}}. As a matter of fact, there are very few papers on the subject of inverse problems for Maxwell equations in the presence of nonlinear materials. They are related to Electrical Resistance Tomography in the special case of a monomial electrical conductivity, i.e. when
\begin{displaymath}
    \sigma(x,E)=\sigma_0(x)E^p(x).
\end{displaymath}
Specifically, in~\cite{art:Sa12, art:Bra14} the Calder\'on problem for the $p$-Laplacian is posed and it is proven that the boundary values of conductivity can be uniquely determined by the DtN operator. In~\cite{art:Ca20}, the authors treat the inverse problem for the electrical conductivity given by a linear term plus a monomial term. In~\cite{art:Bra15, art:Salo16} the Monotonicity Principle (MP) is generalized to the case of monomial conductivity.
The most comprehensive results are those of \cite{art:Co21,MPMETHODS,art:Co23} where the \textcolor{black}{authors} discover a Monotonicity Principle for arbitrary nonlinear materials.

Broadly speaking, the Monotonicity Principle is a very general property which underpins of a class of real-time imaging algorithms, that have been successfully applied to a large variety of problems~\cite{art:Ta02,art:Ta06,art:Ta06p,art:Ca12,art:Ta03,art:Ta10,art:Su17,art:Ta16,art:Su17_2,book:Ta15,art:To20,garde2022reconstruction,garde2022reconstruction,albicker2023monotonicity,albicker2020monotonicity,kar2023fractional,tamburrino2021themonotonicity}. The MP states a monotone relationship between the point-wise value of the unknown material property and a proper boundary operator which can be measured. In turn, this makes it possible to determine whether or not a proper voxel (test domain) of $\Omega$ is part of the unknown anomaly $A$. As well as providing real-time performances, which is a rare feature, the MP provides upper and lower bounds \cite{art:Ha15, art:Ta02, art:Ta16_1} and clear theoretical limits in its performance \cite{albicker2020monotonicity,daimon2020monotonicity,art:Ha13}.

The MP was originally proven in~\cite{art:Gi90} in the field of steady currents problems in linear conductors. Then,~\cite{art:Ta02} recognizes its relevance in inverse problems and a new class of imaging methods and algorithms is proposed.
In~\cite{art:Co21, MPMETHODS, art:Co23}, the MP is extended to nonlinear materials, under very general assumptions. \textcolor{black}{In order to treat nonlinear materials, the authors introduce a new proper boundary operator, called \emph{average} DtN, which reflects the monotonicity of the Dirichlet energy, with respect to the material property}. Nonlinear problems where the boundary data is either large or sufficiently small are investigated in \cite{corboesposito2023thep0laplacesignature,corboesposito2023theplaplacesignature}. Other monotonicity-based reconstruction methods can be found in \cite{garde2022simplified} for piecewise constant layered conductivities and in \cite{arens2023monotonicity} for the Helmholtz equation in a closed cylindrical waveguide with penetrable scattering objects.

This work complements \cite{art:Co21,MPMETHODS, art:Co23} and provides the imaging method and related algorithm for treating nonlinear materials. To the best of our knowledge, this is the first ad-hoc imaging method for nonlinear problems. Realistic numerical examples highlight the performance of the proposed method. The paper is organized as follows: in Section \ref{sec2} the mathematical model is presented, in Section \ref{sec3} the MP for nonlinear materials is briefly reviewed, in Section~\ref{sec4} the imaging method is proposed and discussed, in Section \ref{sec5} realistic numerical examples demonstrate the key features and performance of the method and, finally, in Section \ref{sec13} the conclusions are drawn.

\section{Mathematical Model}\label{sec2}
Let $\Omega\subset\mathbb{R}^n$, \textcolor{black}{$n \ge 2$} be the region under tomographic inspection. We assume $\Omega$ to be an open bounded \textcolor{black}{connected} domain with Lipschitz boundary $\partial\Omega$. $\hat{\mathbf{n}}$ is the outer unit normal on $\partial\Omega$ and $\partial_n$ denotes the outer normal derivative defined on $\partial\Omega$. Furthermore, in the following, $L^{\infty}_+(\Omega)$ and $X_{\diamond}(\partial\Omega)$ are the functional spaces defined as
\begin{equation}
\begin{split}
    L^{\infty}_+(\Omega) & =\{u\in L^{\infty}(\Omega):u\geq c_0>0\text{ a.e. in $\Omega$}\},\\
    X_{\diamond}(\partial\Omega) & =\left\{g\in H^{1/2}(\partial\Omega): \int_{\partial\Omega} g=0\right\}.
\end{split}
\end{equation}

In order to guarantee the well-posedness of the direct problem~\eqref{eqn:dirint}, suitable assumptions on \textcolor{black}{$\gamma_A$} are required. These can be found in~\cite{MPMETHODS}, where the authors considered a more general problem, with both $A$ and $B=\Omega\setminus A$ filled by nonlinear materials. 
In the following we specialize the assumptions in~\cite{MPMETHODS} to the case of interest consisting of a bounded \textcolor{black}{nonlinear material property $\gamma_{nl}$ and a bounded linear material property $\gamma_{bg}$}. 

Specifically, it is required that \textcolor{black}{$\gamma_{bg}\in L^{\infty}_+(\Omega)$}, with
\begin{equation*}
    \textcolor{black}{c_{bg}^l\leq\gamma_{bg}(x)\leq c_{bg}^u},
\end{equation*} 
where $c_{bg}^l$ and $c_{bg}^u$ are two positive constants, and that \textcolor{black}{$\gamma_{nl}:\overline{\Omega}\times[0,+\infty)\to\mathbb{R}$} satisfies the following assumptions:
\begin{enumerate}
\item [{\bf (H1)}]\textcolor{black}{$\gamma_{nl}$} is a Carath\'eodory function.
\item [{\bf (H2)}]\textcolor{black}{$s\in [0,+\infty) \mapsto \gamma_{nl}(x,s)s$} is strictly increasing for a.e. $x\in\overline{\Omega}$.
\item [{\bf (H3)}]There exist two positive constants \textcolor{black}{$c_{nl}^l < c_{nl}^u$} such that
\begin{equation}
\label{eqn:con_mu}
\textcolor{black}{c_{nl}^l\leq\gamma_{nl}(x,s)\leq c_{nl}^u}
\end{equation}
\textcolor{black}{$\text{for a.e.}\ x\in {\overline \Omega}\ \text{and}\ \forall s\geq 0$}.
\item [{\bf (H4)}]There exists $\kappa>0$ such that
\begin{equation*}
\textcolor{black}{(\gamma_{nl}(x,s_2){\bf s}_2-\gamma_{nl}(x,s_1){\bf s}_1)\cdot( {\bf s}_2-{\bf s}_1) \geq \kappa|{\bf s}_2-{\bf s}_1|^2}
\end{equation*}
       \ $\text{for a.e.}\ x\in \Omega$, and for any \textcolor{black}{${\bf s}_1,{\bf s}_2\in\mathbb{R}^n$}.
\end{enumerate}

\textcolor{black}{For the sake of clarity, we point out that, since $\gamma_{nl}$ is a Carath\'eodory function,}
\begin{itemize}
\item \textcolor{black}{$x\in\overline\Omega\mapsto \gamma_{nl}(x,s)$} is measurable for every \textcolor{black}{$s\in[0,+\infty)$},
\item \textcolor{black}{$s\in [0,+\infty)\mapsto \gamma_{nl}(x, s)$} is continuous for almost every $x\in\Omega$.
\end{itemize}

Problem~\eqref{eqn:dirint} can be cast in the weak form as
\begin{equation*}
\label{eqn:weakform}
\textcolor{black}{
    \int_{\Omega}\gamma(x,\abs{\nabla u(x)})\nabla u(x)\cdot \nabla\varphi(x)\,dx=0\quad \forall \varphi \in C^{\infty}_0(\Omega),
    }
\end{equation*}
\textcolor{black}{where $u \in H^1(\Omega)$ and $u\rvert_{\partial\Omega}=f\in X_{\diamond}(\partial\Omega)$}.

Furthermore (see~\cite{MPMETHODS}), the unique weak solution \textcolor{black}{$u$} of problem~\eqref{eqn:weakform} can be variationally characterized as
\begin{equation}
\textcolor{black}{
    u=\argmin\left\{\mathbb{E}(v) : v\in H^1(\Omega),\,  v\rvert_{\partial\Omega}=f\in X_{\diamond}(\partial\Omega)\right\},}
\label{eqn:varprob}
\end{equation}
\textcolor{black}{where} the so-called Dirichlet Energy $\mathbb{E}$ is defined as
\begin{equation*}
\mathbb{E}(u)=\int_{\Omega} Q\left(x,\abs{\nabla u}\right)\,dx,
\end{equation*}
with $Q$ \textcolor{black}{satisfying}
\begin{displaymath}
\textcolor{black}{
    Q(x,s)=\int_0^{s}\gamma(x,\eta)\eta\,d\eta \quad\text{for a.e. $x\in\overline{\Omega}$ and $\forall\,s>0$}.}
\end{displaymath}


\subsection{\textcolor{black}{Connection to physical problems}}
Although all the results of this paper are valid in $\mathbb{R}^n$ with arbitrary $n \in \mathbb{N}$, when $n=2$ or $n=3$, problem \eqref{eqn:dirint} may constitute a magnetostatic\textcolor{black}{, electrostatic or steady currents problem} from a physical standpoint. The connection to physical problems is summarized in Table \ref{tab_01_my_label}.

\begingroup
\renewcommand{\arraystretch}{1.5} 
\begin{table}[htb]
    \centering
    \begin{tabular}{P{0.25\textwidth}P{0.125\textwidth}P{0.235\textwidth}P{0.2\textwidth}}
        \toprule
        Physical model & Material property & Imposed boundary data & Measured boundary data\\
        \midrule
        Magnetostatic & $\mu$ & Scalar magnetic potential & Magnetic flux density \\
        Electrostatic & $\varepsilon$ & Scalar electric potential & Electric flux density \\
        Steady currents & $\sigma$ & Scalar electric potential & Current density \\
        \bottomrule
    \end{tabular}
    \caption{Summary of possible physical models corresponding to equations~\eqref{eqn:dirint} and the related imposed and measured quantities.}
    \label{tab_01_my_label}
\end{table}
\endgroup

\subsubsection{\textcolor{black}{Magnetostatic case}}
\textcolor{black}{Assuming $\Omega$ to be simply connected and free from electrical current densities, and $\gamma=\mu$, where $\mu$ is the nonlinear magnetic permeability, \eqref{eqn:dirint} represents a magnetostatic problem described via the magnetic scalar potential, i.e. \textcolor{black}{$\mathbf{H}=-\nabla u$}, with $\mathbf{H}$ being the magnetic field.}

When $n=3$, it is a \textcolor{black}{\lq\lq classical\rq\rq} magnetostatic problem in three dimensions. On the other hand, when $n=2$, we have a magnetostatic problem under the assumption that the geometry of the domain, the magnetic permeability and the source \textcolor{black}{applied on the boundary} are invariant along one axis, \textcolor{black}{the longitudinal axis ($z-$axis)}. Thus, the analysis can be conducted in a section contained in the $(x,y)$ plane, where the model is again given by problem~\eqref{eqn:dirint}. 

In order to understand the physical meaning of \textcolor{black}{$f=u|_{\partial\Omega}$}, i.e. the value of the scalar magnetic potential imposed on $\partial \Omega$, it is possible to consider the configuration in Figure~\ref{fig_01_domain} where the domain $\Omega$ is surrounded by a material with infinite magnetic permeability and, on $\partial \Omega$, a surface current density $\mathbf{J}_s$ is applied. 

\begin{figure}[htp]
    \centering
\includegraphics[width=0.35\textwidth]{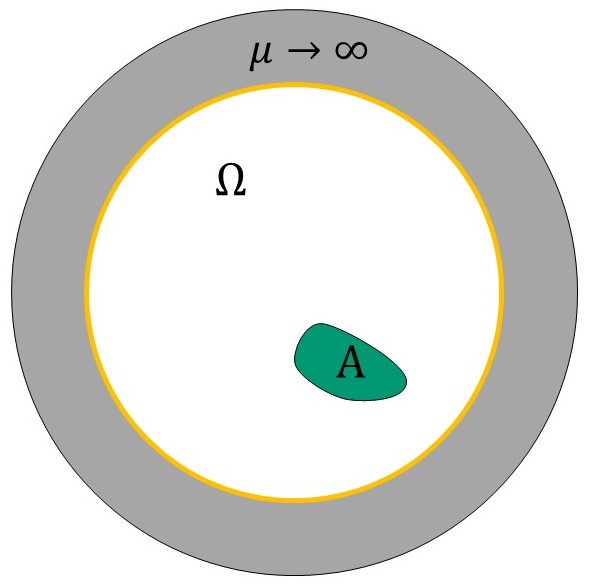}
    \caption{Geometry of the sample under test \textcolor{black}{for the magnetostatic case}. In white the domain $\Omega$, in green the anomaly $A$ and in gray the material surrounding the domain $\Omega$ characterized by an infinite magnetic permeability. \textcolor{black}{A prescribed surface current density $\mathbf{J}_s$ (in yellow) is imposed on $\partial \Omega$.}}
    \label{fig_01_domain}
\end{figure}

The jump condition for the tangential component of the magnetic field, combined with $\mathbf{H}=\mathbf{0}$ in the material with infinite permeability, gives
\begin{equation} \label{eqn:jumpC}
    \hat{\mathbf{n}}\times\mathbf{H}=\mathbf{J}_s.
\end{equation}
Equation \eqref{eqn:jumpC} can be cast in terms of the magnetic scalar potential and the \emph{surface} gradient as \textcolor{black}{$\mathbf{J}_s=-\hat{\mathbf{n}}\times\nabla_s u|_{\partial\Omega}$} or, equivalently, as \textcolor{black}{$\nabla_s u|_{\partial\Omega} = - \mathbf{J}_s \times \hat{\mathbf{n}}$}. Therefore,
\begin{displaymath}
\textcolor{black}{
    f(x)=f(x_0)-\int_{x_0}^x \mathbf{J}_s\times\hat{\mathbf{n}}\cdot\hat{\mathbf{t}}\,\mathrm{d}l,
    }
\end{displaymath}
where the line integral is carried out along an arbitrary curve oriented from $x_0$ to $x$ and lying on  $\partial\Omega$.
\textcolor{black}{$x_0$ is a prescribed reference point and the value of $f(x_0)$ is chosen when imposing the uniqueness of the solution.}

\textcolor{black}{For the magnetostatic problem, the DtN operator $\Lambda$ maps the imposed boundary magnetic potential to the (entering) normal component of the magnetic flux density on $\partial \Omega$}
\begin{displaymath}
    \textcolor{black}{
    \Lambda:f\in X_{\diamond}(\partial\Omega)\to -\mathbf{B}\cdot\hat{\mathbf{n}}|_{\partial\Omega}=\mu\partial_n u|_{\partial\Omega}\in X'_{\diamond}(\partial\Omega).}
\end{displaymath}

\subsubsection{\textcolor{black}{Electrostatic case}}
\textcolor{black}{
Assuming $\gamma=\varepsilon$, where $\varepsilon$ is the nonlinear dielectric permittivity, \eqref{eqn:dirint} represents an electrostatic  problem described via the electric scalar potential, i.e. \textcolor{black}{$\mathbf{E}=-\nabla u$}, with $\mathbf{E}$ being the electric field. As for the previous case, for $n=3$ we have an electrostatic problem in 3D, whereas $n=2$ corresponds to a $z-$invariant problem.}

\textcolor{black}{In this case $f=u\rvert_{\partial\Omega}$ represents the imposed boundary voltage, i.e.}
\begin{equation}\label{eqn:pot_f}
\textcolor{black}{
    f(x)=f(x_0)-\int_{x_0}^x \mathbf{E}\cdot \hat{\mathbf{t}}\,\mathrm{d}l},
\end{equation}
\textcolor{black}{where the line integral is carried out along an arbitrary curve oriented from $x_0 \in \partial \Omega$ to $x \in \partial \Omega$. $x_0$ is a prescribed reference point and the value of $f(x_0)$ is chosen when imposing the uniqueness of the solution.}

\textcolor{black}{For the electrostatic problem, the DtN operator $\Lambda$ maps the imposed boundary electric potential to the (entering) normal component of the electric flux density on $\partial \Omega$}
\begin{displaymath}
    \textcolor{black}{
    \Lambda:f\in X_{\diamond}(\partial\Omega)\to -\mathbf{D}\cdot\hat{\mathbf{n}}|_{\partial\Omega}=\varepsilon\partial_n u|_{\partial\Omega}\in X'_{\diamond}(\partial\Omega).}
\end{displaymath}

\subsubsection{\textcolor{black}{Steady currents case}}
\textcolor{black}{
Finally, assuming $\gamma=\sigma$, where $\sigma$ is the nonlinear electrical conductivity, \eqref{eqn:dirint} represents a steady currents problem described via the electric scalar potential, i.e. \textcolor{black}{$\mathbf{E}=-\nabla u$}, with $\mathbf{E}$ being the electric field.
The boundary data $f$ assumes the same meaning as~\eqref{eqn:pot_f}.}

\textcolor{black}{For the steady currents problem, the DtN operator $\Lambda$ maps the imposed boundary electric potential onto the (entering) normal component of the current density on $\partial \Omega$}
\begin{displaymath}
    \textcolor{black}{
    \Lambda:f\in X_{\diamond}(\partial\Omega)\to -\mathbf{J}\cdot\hat{\mathbf{n}}|_{\partial\Omega}=\sigma\partial_n u|_{\partial\Omega}\in X'_{\diamond}(\partial\Omega).}
\end{displaymath}

\section{Imaging Method}\label{sec3}
\subsection{The MP for the inverse obstacle problem}
\textcolor{black}{Hereafter, we adopt the following.}
\begin{nota}
\textcolor{black}{
    The DtN operator corresponding to material property $\gamma_a^b$, is denoted by $\Lambda$ equipped with the same superscripts and subscripts, i.e. by symbol $\Lambda_a^b$. In other words, symbol $\gamma$ is replaced by $\Lambda$. The rule applies also when no superscripts or subscripts are present and extends to the average DtN operator, defined in the following.
    }
\end{nota}
The following definitions are adopted from \cite{art:Co21}.
\begin{defn}\label{def:3}
    Let \textcolor{black}{$\gamma$} be a (nonlinear) \textcolor{black}{material property} defined on $\Omega$ and let \textcolor{black}{$\Lambda$} be the related DtN operator. The average DtN operator \textcolor{black}{$\overline{\Lambda}$} is defined as
\begin{displaymath}
    \textcolor{black}{\overline{\Lambda}:f\in X_{\diamond}(\partial\Omega) \to \int_0^1\Lambda(\alpha f)\,d\alpha \in X'_{\diamond}(\partial\Omega)},
\end{displaymath}
where
\begin{equation*}
\textcolor{black}{
    \langle\overline{\Lambda}(f),\varphi\rangle=\int_0^1 \langle \Lambda(\alpha f),\varphi\rangle\,d\alpha\quad \forall \varphi\in X_{\diamond}(\partial\Omega)}.
\end{equation*}
\end{defn}
\begin{defn}
    Inequality \textcolor{black}{$\gamma_1\leq\gamma_2$} means that
    \begin{equation*}
    \textcolor{black}{
        \gamma_1(x,s)\leq\gamma_2(x,s) \text{ for a.e. } x\in\overline{\Omega} \text{ and }\forall\, s>0}
    \end{equation*}
\end{defn}
\begin{defn}\label{def:def3}
    Inequality $\overline{\Lambda}_{1}\leqslant\overline{\Lambda}_{2}$ means that
    \begin{equation*}
    \left\langle \overline{\Lambda}_{1}(f), f\right\rangle\leq\left\langle \overline{\Lambda}_{2}(f), f\right\rangle, \, \forall f\in X_{\diamond}(\partial\Omega).
    \end{equation*}
\end{defn}

\begin{thm}[Monotonicity Principle,~\cite{art:Co21, MPMETHODS}]\label{th:MP}
\textcolor{black}{Let \textcolor{black}{$\gamma_1$ and $\gamma_2$ be two material properties} satisfying (H1)-(H4), then} 
\begin{equation}\label{eqn:MP1}
\textcolor{black}{
    \gamma_1\leq\gamma_2 \Longrightarrow \overline{\Lambda}_{1}\leqslant\overline{\Lambda}_{2}}.
\end{equation}
\end{thm}
\textcolor{black}{The relationship~\eqref{eqn:MP1} expresses the Monotonicity Principle between material properties and boundary data. In particular, it states that if the material property is increased at any point of the domain and for any value of $\lvert \nabla u \rvert$, then the average DtN increases. Moreover, as a boundary operator, the average DtN can be measured from the boundary of the domain only. In this way, an \lq\lq internal\rq\rq \ condition can be detected from boundary data.}

\textcolor{black}{Let $A \subset \subset\Omega$ be the unknown anomaly occupied by the nonlinear material. The anomaly $A$ is well contained in $\Omega$, i.e. is assumed to be at a non vanishing distance from the boundary of $\Omega$. In order to turn \eqref{eqn:MP1} into an imaging method for the inverse obstacle problem, it is necessary to introduce the concept of test anomaly. The test anomaly $T \subset \subset \Omega$ is nothing but a proper anomaly occupying a known region and such that}
\begin{equation}\label{eqn:MP12}
     \textcolor{black}{T \subseteq A \Longleftrightarrow \gamma_T \leq \gamma_A.}
 \end{equation}

\textcolor{black}{Combining \eqref{eqn:MP1} and \eqref{eqn:MP12} gives}
\begin{equation}\label{eqn:MP11}
     \textcolor{black}{T \subseteq A \Longrightarrow \overline{\Lambda}_T \leqslant \overline{\Lambda}_A.} 
 \end{equation}

Starting from this proposition, it is possible to develop an imaging method specifically designed for the inverse obstacle problem, as discussed in the next subsection.

\subsection{Imaging Method}\label{sec:ia}
Equation \eqref{eqn:MP11} is equivalent to
\begin{equation}
\label{eqn:MP13}
\overline{\Lambda}_{T}\not\leqslant\overline{\Lambda}_{A} \Longrightarrow T\not\subseteq A,
\end{equation}
where $\overline{\Lambda}_{T}\not\leqslant\overline{\Lambda}_{A}$ is understood as
\begin{equation}\label{eqn:MP3}
     \exists f\in X_{\diamond}(\partial\Omega) : \langle \overline{\Lambda}_{T}(f),f\rangle> \langle \overline{\Lambda}_{A}(f),f\rangle.
\end{equation}

Equation~\eqref{eqn:MP13}, first proposed in \cite{art:Ta02} for linear materials, underpins the imaging method. Indeed, it makes it possible to infer some information on the behaviour of the unknown magnetic permeability in the interior of $\Omega$, starting from the boundary data. Specifically, relation~\eqref{eqn:MP13} makes it possible to establish whether a test anomaly $T$ is not completely included in the actual anomaly $A$, starting from the knowledge of the boundary data $\overline{\Lambda}_A$ and $\overline{\Lambda}_T$.

Let $\{ T_k \}_k$ be a covering of the region of interest (ROI). The ROI is contained or equal to $\Omega$. By evaluating the elementary test of~\eqref{eqn:MP3} on each $T_k$, it is possible to discard most of or all the $T_k$s that are not completely contained in the unknown anomaly $A$. Indeed, if $\overline{\Lambda}_{T_k}\not\leqslant\overline{\Lambda}_A$ , then we exclude $T_k$ from contributing to the estimate $A^U$ of the unknown anomaly $A$. The basic reconstruction scheme (see \cite{art:Ta02}) is
 \begin{equation}\label{eqn:alg}
    A^U=\bigcup_{k}\{T_k | \overline{\Lambda}_{T_k}\leqslant\overline{\Lambda}_A\}.
 \end{equation}
\begin{rem}
If $A$ is the union of some of the $T_k$s, then $A \subseteq A^U$ (see \cite{art:Ta02, art:Ha15, art:Ta16_1}).
\end{rem}

Although the inversion algorithm follows in fairly simply from the Monotonicity Principle, the practical implementation of the reconstruction rule in~\eqref{eqn:alg} provides some important challenges, especially for treating nonlinear materials. Indeed, for linear materials, it can be easily proven that
\begin{equation}
\label{eqn:LinDtN}
    \overline{\Lambda}=\frac{1}{2}\Lambda,
\end{equation}
where $\Lambda$ is the classical DtN operator, i.e. a linear one.
Therefore, condition~\eqref{eqn:MP3} is true if $\Lambda_{A}-\Lambda_{T_k}$ has at least one negative eigenvalue. On the contrary, for the nonlinear case, condition~\eqref{eqn:MP3} does not correspond to an eigenvalue problem. No general mathematical tools are available to establish the existence of a boundary potential satisfying~\eqref{eqn:MP3}.

\subsection{\textcolor{black}{Design of the Test Anomalies}}
\textcolor{black}{In designing the test anomalies, it is necessary to distinguish two possible cases because the material property of the background may be well separated from that of the nonlinear material (Figure \ref{fig_02_sepa} and \ref{fig_02_sepb}) or not well separated (Figure \ref{fig_02_sepc})}.
\begin{figure}[htp]
    \centering
    \subfloat[][\label{fig_02_sepa}]
    {\includegraphics[width=.45\textwidth]{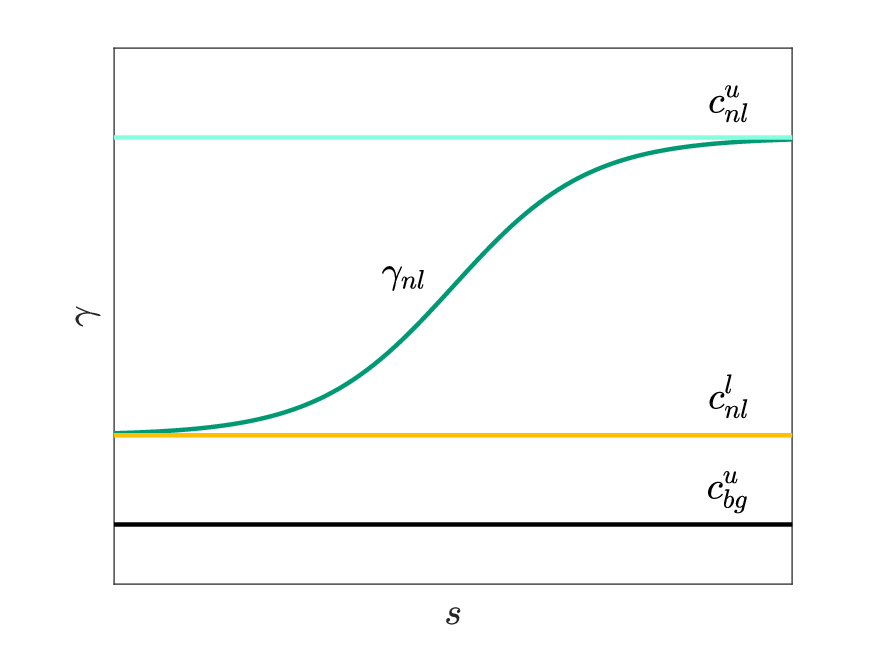}} \quad
    \subfloat[][\label{fig_02_sepb}]
    {\includegraphics[width=.45\textwidth]{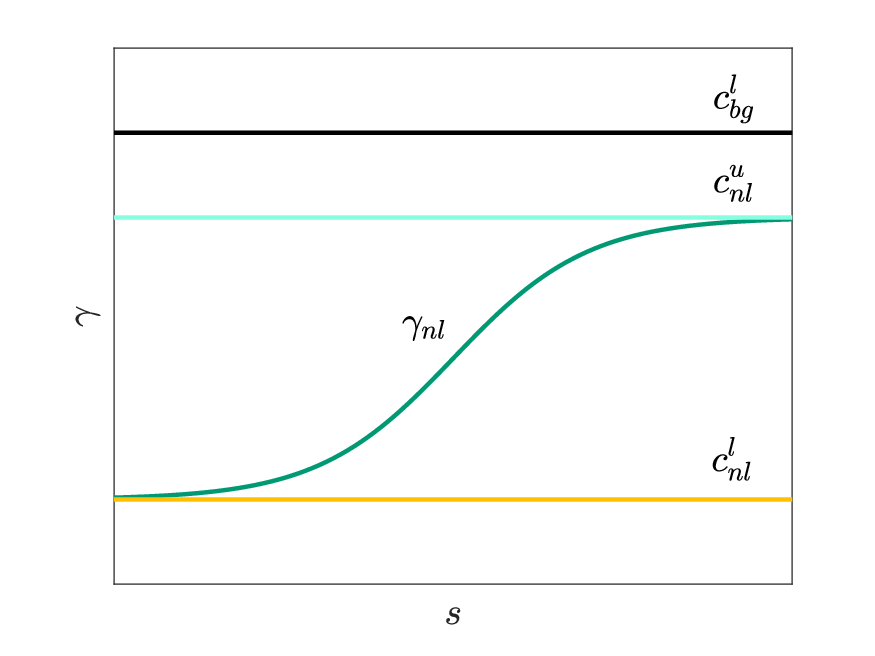}} \\
    \subfloat[][\label{fig_02_sepc}]
    {\includegraphics[width=.45\textwidth]{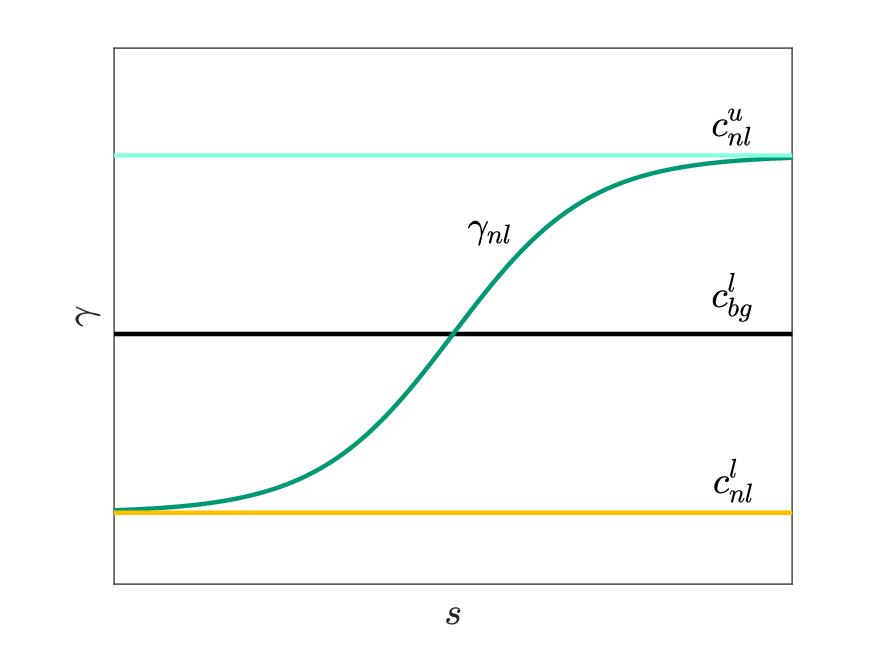}} 
    \caption{\textcolor{black}{A: well separated ($c_{nl}^l>c_{bg}^u$) material properties. B: well-separated ($c_{nl}^u<c_{bg}^l$) material properties. C: not well separated material properties.}}
\end{figure}

\subsubsection{Well separated material properties}
\textcolor{black}{In this case either $c_{nl}^l > c_{bg}^u$ or $c_{bg}^l > c_{nl}^u$. In this contribution we consider the former case $(c_{nl}^l > c_{bg}^u)$ only, because the latter one  can be treated similarly.}

\textcolor{black}{A proper option for defining the material property for a test anomaly in $T$ is}
\begin{equation}\label{eqn:t1}
\textcolor{black}{
    \gamma_T(x,s)=\begin{cases}
        \gamma_{nl}(x,s) & \text{in $T$} \\
        \gamma_{bg}(x) & \text{in $\Omega\setminus T$}.
    \end{cases}
    }
\end{equation}
\textcolor{black}{In other words, the test anomaly is given by the same material property as the unknown anomaly.}

\textcolor{black}{This definition for $\gamma_T$ satisfies condition \eqref{eqn:MP12}, as required.}

\subsubsection{\textcolor{black}{Not well separated material properties}}
\label{subs:notwellsep}
\textcolor{black}{When $c_{nl}^l < c_{bg}^u < c_{nl}^u$, the material property for a test anomaly in $T$ is defined as follows
\begin{equation}\label{eqn:t2}
\textcolor{black}{
    \gamma_T(x,s)=\begin{cases}
            \gamma_{nl}(x,s) & \text{in $T$} \\
            \min({\gamma_{bg}(x),\gamma_{nl}(x,s)})    & \text{in $\Omega\setminus T$}.
    \end{cases}}
\end{equation}}

\textcolor{black}{This definition for $\gamma_T$ satisfies condition \eqref{eqn:MP12}, as required.}

\textcolor{black}{The dual case of $c_{nl}^l < c_{bg}^l < c_{nl}^u$, can be treated similarly.}


\section{Testing condition $\overline{\Lambda}_{T}\nleqslant\overline{\Lambda}_{A}$}\label{sec4}
Hereafter, we assume that $c_{bg}^u < c_{nl}^l$. A similar treatment can be developed when \textcolor{black}{$c_{nl}^u < c_{bg}^l$}.

As highlighted in the previous section, condition~\eqref{eqn:MP13} represents the elementary monotonicity test, making it possible to infer whether or not a test anomaly $T$ is contained in the anomaly $A$ from the knowledge of the average DtN operators, i.e. from $\overline{\Lambda}_T$ and the measured data $\overline{\Lambda}_A$. The major challenge, in setting up a Monotonicity Principle based inversion method, is given by the search for the proper boundary data $f$, if any, such that $\langle \overline{\Lambda}_{T}(f),f\rangle> \langle \overline{\Lambda}_{A}(f),f\rangle$. This problem is challenging because of the nonlinear nature of operators $\overline{\Lambda}_A$ and $\overline{\Lambda}_T$.

A na\"ive idea is to pose the problem in terms of a nonquadratic minimization
\begin{equation}
\label{eqn:iter_prob}
    \min_{f}\langle \overline{\Lambda}_A(f)-\overline{\Lambda}_T(f), f\rangle < 0,
\end{equation}
where the aim is to verify whether or not the minimum of the functional $\langle \overline{\Lambda}_A(f)-\overline{\Lambda}_T(f), f\rangle$ is negative. Unfortunately, this approach is not at all practical and may require a huge number of iterations, and hence measurements, with an execution time clearly incompatible with real-world applications.

In the following, we propose a systematic approach to select an a priori set of \emph{proper} boundary data $f$ to verify whether or not $\overline{\Lambda}_{T}\not\leqslant\overline{\Lambda}_{A}$ is true. These candidate \lq\lq test\rq\rq\  boundary data are evaluated before the measurements, thus preserving the compatibility of the proposed method with real-time applications.

\subsection{Basic idea and main result}
\label{sec:basid}
As mentioned above, the aim is to find a set of boundary potentials $f$ that are able to reveal whether a test anomaly $T$ is not included in $A$, i.e. such that
\begin{displaymath}
    \langle \overline{\Lambda}_A(f)-\overline{\Lambda}_T(f),f\rangle<0 \ \ \text{for $T\not\subseteq A$}.
\end{displaymath}
What makes the problem difficult is that these boundary potentials have to be computed in advance and before the measurement process takes place, i.e. without any knowledge of $A$.

The key for evaluating these potentials entails finding a quadratic upper bound to both $\langle \overline{\Lambda}_A(f),f\rangle$ and $\langle -\overline{\Lambda}_T(f),f\rangle$. Once this quadratic upper bound is available, the required boundary potential can be evaluated via eigenvalue computation.

\subsubsection{Upper bound to $\langle \overline{\Lambda}_A(f),f\rangle$}
Let $F$ be a known domain. When $A \subseteq F$, we have (see Figure~\ref{fig_03_dis1})
\begin{equation} \label{eqn:ineq}
\textcolor{black}{\gamma_A \le \gamma_F \le \gamma_F^u},    
\end{equation}
where
\begin{equation}\label{eqn:mufl}
\textcolor{black}{
    \gamma_F^u(x)=\begin{cases}
        c_{nl}^u & \text{in $F$}\\
        \gamma_{bg}(x) & \text{in $\Omega\setminus F$}.
    \end{cases}
    }
\end{equation}

\begin{figure}[htp]
\centering
\subfloat[][]
{\includegraphics[width=.3\textwidth]{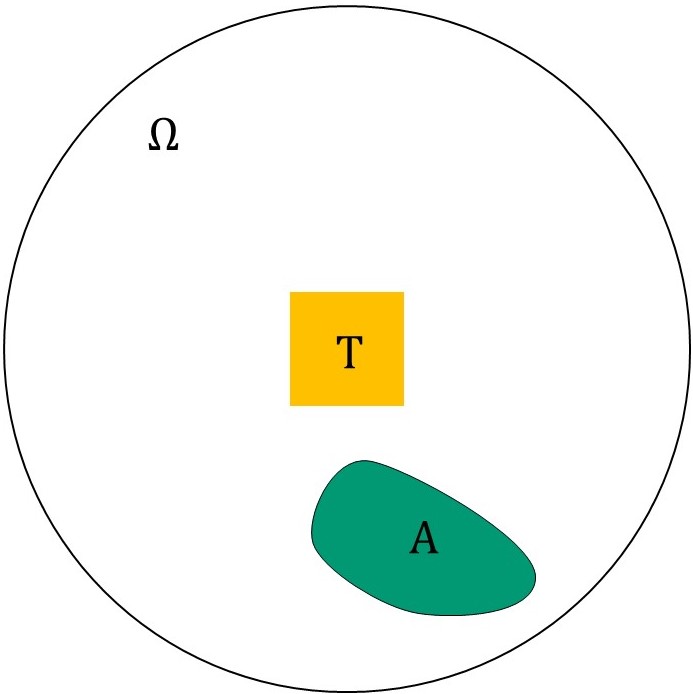}} \quad
\subfloat[][]
{\includegraphics[width=.3\textwidth]{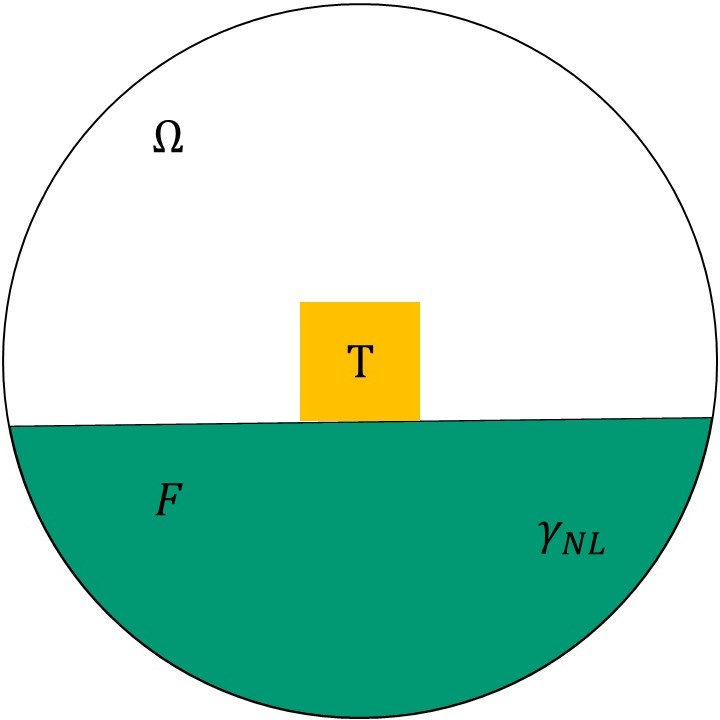}}\quad
\subfloat[][]
{\includegraphics[width=.3\textwidth]{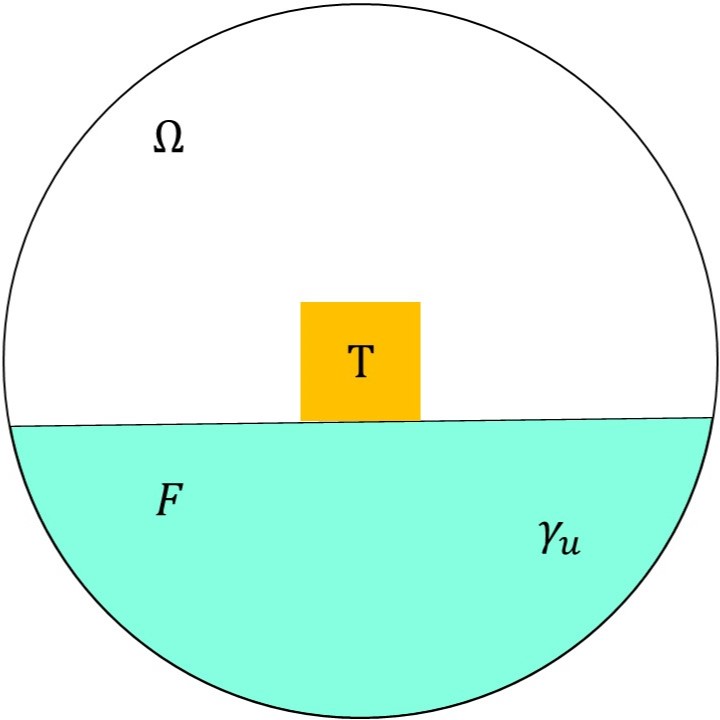}}
\caption{Left: the actual configuration. Center: the real anomaly is included in a known anomaly $F$. Right: the nonlinear material in $F$ is replaced by a linear one.}
\label{fig_03_dis1}
\end{figure}

As a consequence, the Monotonicity Principle applied to \eqref{eqn:ineq}, supplies the following inequality chain
\begin{equation}
\label{eqn:firstIneq}
    \overline{\Lambda}_A\leq\overline{\Lambda}_F\leq\overline{\Lambda}_F^u,
\end{equation}
where $\overline{\Lambda}_F^u$ is the average DtN operator for $\gamma_F^u$.

Equation \eqref{eqn:firstIneq} provides the required quadratic upper bound
\begin{equation}\label{eqn:firstIneq2}
    \langle \overline{\Lambda}_A(f),f\rangle \le \langle \overline{\Lambda}_F^u f,f\rangle, \ \forall f \in X_{\diamond}(\partial \Omega),
\end{equation}
because $\overline{\Lambda}_F^u$ is a linear operator since the underlying material property $\gamma_F^u$ is linear. Indeed, \textcolor{black}{$\gamma_F^u$} corresponds to an anomaly in $F$ containing the value of the upper bound to \textcolor{black}{$\gamma_{nl}$} as a \textcolor{black}{material property}.

\subsubsection{Upper bound to $\langle -\overline{\Lambda}_T(f),f\rangle$}
We have
\begin{equation}
\label{eqn:lowbounds}
    \textcolor{black}{\gamma_T \ge \gamma_T^l},    
\end{equation}
where
\begin{equation}\label{eqn:mut}
 \textcolor{black}{\gamma_{T}^l(x)=\begin{cases}
        c_{nl}^l & \text{in $T$}, \\
        \gamma_{bg}(x) & \text{in $\Omega\setminus T$}.
    \end{cases}}
\end{equation}
The material property \textcolor{black}{$\gamma_{T}^l$} corresponds to a linear material where the anomalous region $T$ is filled by $c_{nl}^l$, a lower bound to $\gamma_{nl}$ (see Figure \ref{fig_04_fict}). 
Finally, the Monotonicity Principle applied to \eqref{eqn:lowbounds}, yields the following inequality
\begin{equation}
\label{eqn:secondIneq}
    -\overline{\Lambda}_T \leq -\overline{\Lambda}_T^l,
\end{equation}
i.e.
\begin{equation}\label{eqn:secondIneq2}
    -\langle \overline{\Lambda}_T(f),f\rangle \le - \langle \overline{\Lambda}_T^l f,f\rangle, \ \forall f \in X_{\diamond}(\partial \Omega).
\end{equation}


\begin{figure}[htp]
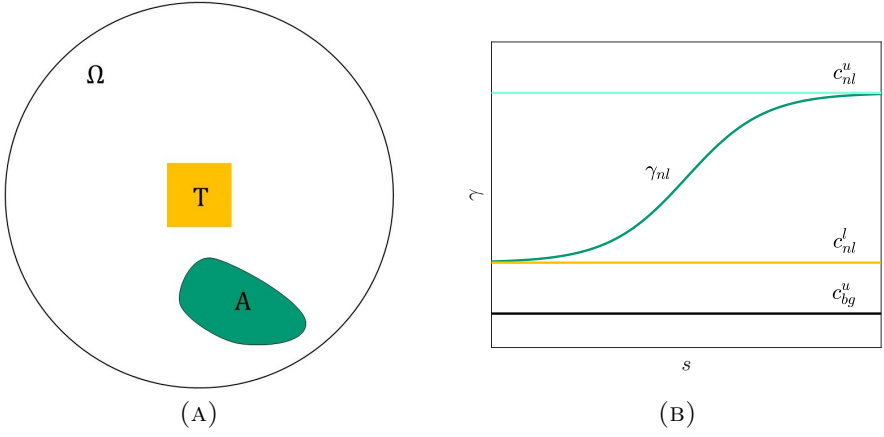

\centering
\subfloat[][]
{\includegraphics[width=.35\textwidth]{M03_TestEAn.jpg}} \quad
\subfloat[][]
{\includegraphics[width=.45\textwidth]{M02_Fig2_a.eps}}
\caption{Left: Test anomaly not included in $A$. Right: \textcolor{black}{Material properties involved.}}
\label{fig_04_fict}
\end{figure}

\subsubsection{Boundary data}
In this Section we compute the boundary data able to reveal whether $T \not \subseteq A$, under the assumption that $A \subseteq F$, where $F$ is given and $A$ unknown.

From \eqref{eqn:firstIneq2}, \eqref{eqn:secondIneq2} and \eqref{eqn:LinDtN} we have
\begin{equation}
\begin{split}
    \langle  \overline{\Lambda}_A(f)- \overline{\Lambda}_T(f),f \rangle & \leq
    \langle  \overline{\Lambda}_F^u f- \overline{\Lambda}_T^l f,f \rangle \\
    & \leq
    \frac{1}{2}\langle  \Lambda_F^u f- \Lambda_T^l f,f \rangle.    
\end{split}
\end{equation}
As a result, if there exists boundary data $f$ such that
\begin{equation}
    \langle  \Lambda_F^u f- \Lambda_T^l f,f \rangle < 0,
\end{equation}
then $f$ ensures
\begin{equation}
    \langle  \overline{\Lambda}_A(f)- \overline{\Lambda}_T(f),f \rangle<0.
\end{equation}
The operator $\Lambda_F^u- \Lambda_T^l$ is linear and, hence, the boundary data $f$ is an arbitrary element of the span of the eigenfunctions corresponding to negative eigenvalues of this operator.

We have thus proved the following results.
\begin{prop}\label{th:th1}
    Let $A\subset\Omega$ be an anomalous region characterized by the nonlinear \textcolor{black}{material property}~\eqref{eqn:mua}, assuming $c_{bg}^u < c_{nl}^l$. Let $T\subset\Omega$ be the region occupied by a test anomaly and let $F\subset\Omega$ be such that (i) $A\subseteq F$ and (ii) that $T \not\subseteq F$. If $\Lambda_F^u- \Lambda_T^l$ has negative eigenvalues, then any element $f\in X_{\diamond}(\partial\Omega)$ of the span of the eigenfunctions of the negative eigenvalues guarantees that
    \begin{equation}
        \label{ADtN_f_<0}
        \langle\overline{\Lambda}_{A}(f)-\overline{\Lambda}_T(f),f\rangle<0.
    \end{equation}
\end{prop}
In other words, the boundary data of Proposition~\ref{th:th1} guarantees the possibility to reveal that $T$ is not inclued in $A$ and, therefore, that it does not contribute to the set union appearing in reconstruction formula \eqref{eqn:alg}.



      

Proposition \ref{th:th1} is significant because it shifts the problem from an iterative one (see \eqref{eqn:iter_prob}), requiring proper measurement to be carried out during the minimization process, to a noniterative one, where only the eigenvalues of a proper linear operator $\Lambda_F^u- \Lambda_T^l$ are required. This operation affords a significant advantage, as we can perform all the calculations before the measurement process, which is a fundamental requirement to achieve real-time performance, and the applied boundary potentials during the measurement process are a priori known.

\subsubsection{The Imaging method revisited}
At this stage, one last issue remains to be faced: the design of the fictitious anomalies $F$. Specifically, for each test anomaly $T$, we have to introduce a set of fictitious anomalies $\{F_j\}_j$ such that if $T\nsubseteq A$, then $A$ is completely included in at least one of $F_j$. This issue is treated in a dedicated Section (see Section \ref{sec:fian}
).

Summing up, the proposed method involves the following steps
\begin{itemize}
    \item Once the geometry of the domain $\Omega$ is available, a set of test anomalies $\{T_i\}_i$ is defined.
    \item A set of fictitious anomalies $\{F_{i,j}\}_j$ is defined for each test anomaly $T_i$. Furthermore, given \textcolor{black}{$\gamma_{nl}$ and $\gamma_{bg}$}, a set of test boundary potentials $\{f_{i,j,k}\}_k$ is selected from the span of the eigenfunctions corresponding to negative eigenvalues of $\Lambda_{F_{i,j}}^u-\Lambda_{T_i}^l$.
    \item The values of the operators $\overline{\Lambda}_{T_i}$, evaluated on the potentials $f_{i,j,k}$, are pre-computed and stored once and for all, before any measurement. 
    \item The operator $\overline{\Lambda}_A$ is evaluated on the boundary potentials $f_{k,i}$, for each $i$, $j$ and $k$.
    \item For each test anomaly $T_i$, the values of the differences 
    \[
    \langle\overline{\Lambda}_A(f_{i,j,k})-\overline{\Lambda}_{T_i}(f_{i,j,k}),f_{i,j,k}\rangle
    \]
    are computed. If, given $i$, there exists $j$ and $k$ such that 
    \[
    \langle\overline{\Lambda}_{A}(f_{i,j,k})-\overline{\Lambda}_{T_i}(f_{i,j,k}),f_{i,j,k}\rangle<0,
    \]then the test anomaly $T_i$ is not included in the reconstruction $A^U_D$ of $A$.
    \item The reconstruction is given by the union of test anomalies $T_i$ such that the difference $\langle\overline{\Lambda}_{A}(f_{i,j,k})-\overline{\Lambda}_{T_i}(f_{i,,j,k}),f_{i,j,k}\rangle$ is positive for each $j$ and $k$, i.e.
    \begin{equation}\label{eqn:r0}
        A^U_D=\bigcup_i \{T_i | \langle\overline{\Lambda}_{A}(f_{i,j,k}),f_{i,j,k}\rangle-\langle\overline{\Lambda}_{T_i}(f_{i,j,k}),f_{i,j,k}\rangle\geq0 \ \forall\,j,k\}.
    \end{equation}
\end{itemize}

\begin{rem}
The proposed approach is designed to be compatible with real-time applications. Indeed, the computational cost related to the operations following the measurement process, is given by only the comparison between the pre-computed and stored values of $\langle\overline{\Lambda}_{T_i}(f_{i,j,k}),f_{i,j,k}\rangle$ and the values $\langle\overline{\Lambda}_{A}(f_{i,j,k}),f_{i,j,k}\rangle$ coming from the measurements. Furthermore, the number of required measurements grows linearly with the number of test anomalies.
\end{rem}
\begin{rem}
    Computing the test boundary potentials $f_{i,j,k}$ only requires the knowledge of \textcolor{black}{$\gamma_{bg}(x)$, $c_{nl}^u$ and $c_{nl}^l$}. In other words, we do not need to explicitly know the behaviour of \textcolor{black}{$\gamma_{nl}$} in order to carry out reconstructions, but only its upper and lower bounds.
\end{rem}

\subsection{Nonlinear magnetic permeability intersecting the background permeability}
\label{sec:advid}
Here the case when $c_{nl}^l < c_{bg}^u < c_{nl}^u$, shown in Figure \ref{fig_05_perm} is treated. The test anomaly for this case is given in \eqref{eqn:t2}.
\begin{figure}[htp]
    \centering
    \includegraphics[width=\textwidth]{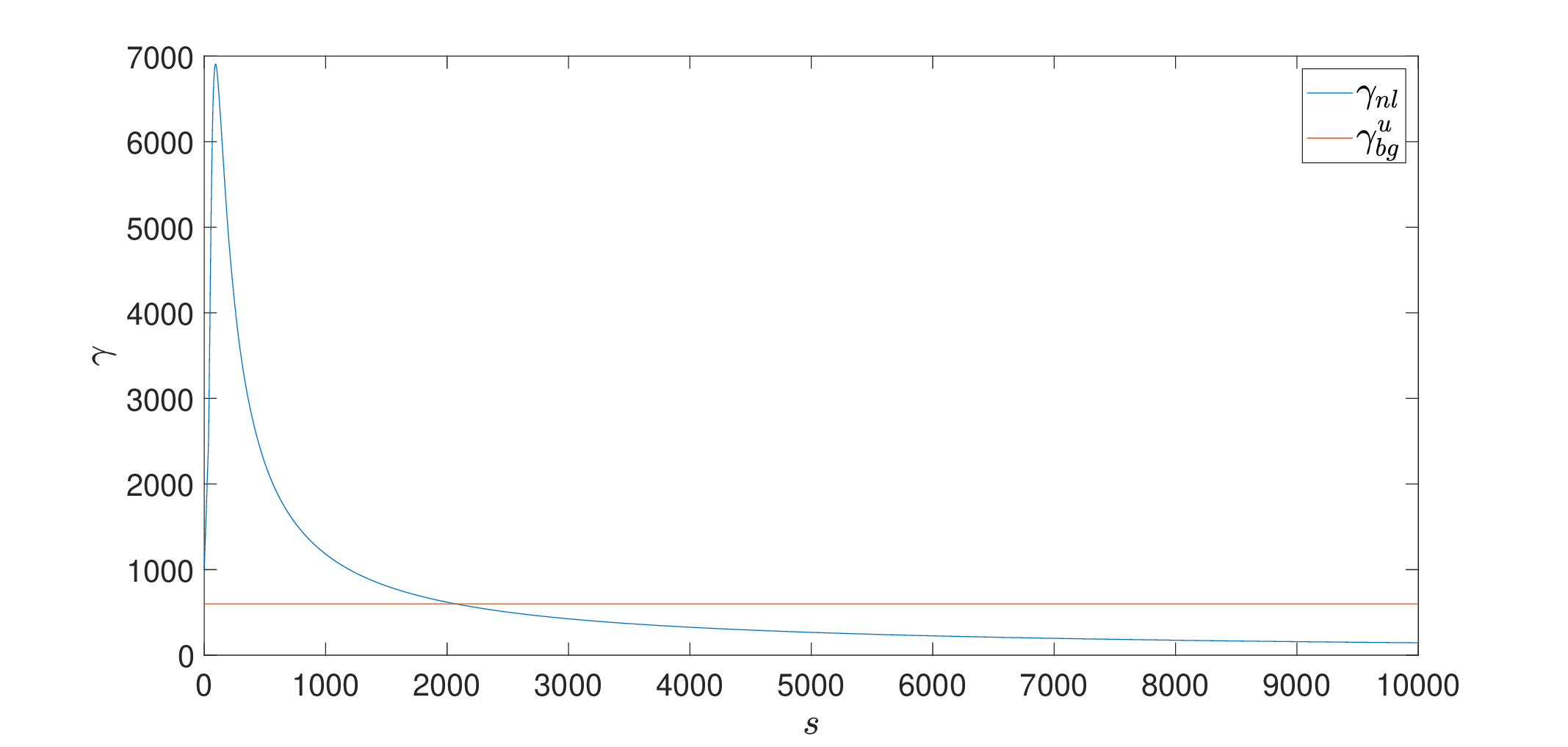}
    \caption{\textcolor{black}{Not well separated material properties. The non linear material property $\gamma$ considered here is the relative magnetic permeability of M$330\mhyphen50$ Electrical Steel. The horizontal axis corresponds to the magnetic field magnitude $H$, i.e. $s=H/(1 \text{A/m})$.}}
    \label{fig_05_perm}
\end{figure}


The focus of this Section, in line with Section \ref{sec:basid}, is to evaluate a negative upper bound to $\langle \overline{\Lambda}_A(f)-\overline{\Lambda}_T(f),f\rangle$. In doing this, the upper bound of $\langle \overline{\Lambda}_A(f),f\rangle$ can be found with the same arguments as Section \ref{sec:basid}, while the upper bound for $\langle -\overline{\Lambda}_T(f),f\rangle$, which involves a redefined \textcolor{black}{material property $\gamma_T$}, requires a new treatment.

In Section~\ref{sec:basid}, a convenient upper bound to $\langle -\overline{\Lambda}_T(f),f\rangle$ is obtained by replacing the nonlinear \textcolor{black}{material described by $\gamma_T$} with the linear one described by \textcolor{black}{$\gamma_T^l$}. The key requirement to make this replacement is that the \textcolor{black}{$\gamma_T$} has to be linear, if restricted to $\Omega\setminus T$. This is not the case of~\eqref{eqn:t2}, thus making the problem more complex. 

\textcolor{black}{In the following, for the sake of clarity, we consider a nonlinear anomaly of the form $\gamma_{nl}(x,s)=\gamma_{nl}(s)$. In any case, the results of this subsection are valid also when the nonlinear material is not spatially uniform.}

\textcolor{black}{Let us consider the two material properties of Figure~\ref{fig_05_perm}. Let $s_0$ be the intersection between $\gamma_{nl}$ and $\gamma_{bg}$. $s_0$ divides the horizontal axis into two regions: one to the left of the intersection point $s_0$, and one to the right of $s_0$. In each region the material properties are ordered, i.e. either $\gamma_{nl} > c_{bg}^u$ or $\gamma_{nl} < c_{bg}^u$.}

\textcolor{black}{
The na\"ive idea to treat the case of intersecting material properties, entails driving the system at a \emph{sufficiently} small boundary potentials so that $s<s_0$. Consequently, the material properties are ordered and it is possible to proceed as in the previous subsection. Indeed, if $s_M$ is a positive value lower than $s_0$, then, for $s\in[0,s_M]$, the lower bound $\gamma_l$ to the nonlinear material property $\gamma_{nl}$ is strictly greater than $c_{bg}^u$. Then, the desired upper bound to $\langle \overline{\Lambda}_T(f),f\rangle$ is $\langle \overline{\Lambda}_T^l f,f\rangle$, where $\overline{\Lambda}_T^l$ is the average DtN operator related to the linear material property}
\begin{equation}\label{eqn:mut5}
    \textcolor{black}{\gamma_{T}^l(x)=\begin{cases}
        \gamma_l & \text{in $T$}, \\
        \gamma_{bg}(x) & \text{in $\Omega\setminus T$},
    \end{cases}}
\end{equation}
\textcolor{black}{and $\gamma_l$ is equal to $\min_{s\in[0,s_M]}\gamma_{nl}(s)$ in this case}.


\textcolor{black}{In the following we prove (see Proposition \ref{th:th3}) that $\overline{\Lambda}_T^l$ is essential to provide the potentials to reveal that $T\not\subseteq A$ when $A \subset F$.}

\textcolor{black}{Before stating the main result, a preliminary lemma is proven which ensures that as the amplitude of the applied boundary potential is sufficiently small, $\langle \overline{\Lambda}_T(f),f\rangle$ approaches a value greater than or equal to $\langle \overline{\Lambda}_{T}^l f,f\rangle$.}

\begin{lem}\label{th:th2_v}
    \textcolor{black}{Let $c_{bg}^u$ be the upper bound to the linear material property $\gamma_{bg}$ and $\gamma_{nl}(s)$ be a nonlinear material property satisfying (H1)-(H4). Let $0<s_M<s_0$ and} 
    \begin{equation*}
            \textcolor{black}{
            \gamma_l=\min_{s\in[0,s_M]}\gamma_{nl}(s).} 
    \end{equation*}
    \textcolor{black}{Let $\mu_T$ and $\mu_{T}^l$ be defined as in~\eqref{eqn:t2} and in~\eqref{eqn:mut5}, and let $\overline\Lambda_T$, $\Lambda_T$, $\overline\Lambda_T^l$, $\Lambda_T^l$ be the corresponding average DtN and classical DtN, then}
    \begin{equation*}
        \lim_{\lambda\to 0^+}\frac{\langle\overline{\Lambda}_T(\lambda f),\lambda f\rangle}{\lambda^2} \geq \langle \overline{\Lambda}_{T}^l f,f\rangle \quad \forall f\in X_\diamond(\partial\Omega).    
    \end{equation*}
\end{lem}
\begin{proof}
\textcolor{black}{Let $\gamma^{lim}$ be the limiting ($s \to 0$) material property defined as}
    \begin{equation*}
    \textcolor{black}{\gamma^{lim}(x)=\begin{cases}
        \gamma_{nl}(0) & \text{in $T$}, \\
        \gamma_{bg}(x) & \text{in $\Omega\setminus T$},
    \end{cases}}
    \end{equation*}
    \textcolor{black}{and let $\overline{\Lambda}^{lim}$ be the related average DtN operator.}    
    \textcolor{black}{A straightforward consequence of equation (5.9) in \cite[Th. 5.1]{corboesposito2023thep0laplacesignature} applied with the following setting: $B=\Omega$, $A=\emptyset$, $Q_B(x,H)=\int_0^H\gamma_T(x,\xi)\xi\,d\xi$ is}
    \begin{equation*}
        \begin{split}
            \lim_{\lambda\to 0^+}\frac{\langle\overline{\Lambda}_T(\lambda f),\lambda f\rangle}{\lambda^2} & =\lim_{\lambda\to 0^+}\frac{1}{2}\frac{\langle{\Lambda}_T(\lambda f),\lambda f\rangle}{\lambda^2} \\
            & =\langle \overline{\Lambda}^{lim} f,f\rangle\geq \langle \overline{\Lambda}_T^{l} f,f\rangle \quad \forall f\in X_\diamond(\partial\Omega).
        \end{split}     
    \end{equation*}
    \textcolor{black}{The second inequality follows from $\gamma^{lim}\geq\gamma_T^l$ and the Monotonicity Principle of Theorem \ref{th:MP}.}
\end{proof}


The following Proposition proves that the linear operator $\Lambda_F^u- \Lambda_T^l$ provides the proper boundary potentials to reveal that $T\not\subseteq A$ when (i) $A \subset F$ and (ii) $T \not\subseteq F$.

\begin{prop}
\label{th:th3}
Let $A\subset\Omega$ be an anomalous region characterized by the nonlinear material property~\eqref{eqn:mua}. Let $T\subset\Omega$ be the region occupied by a test anomaly and let $F\subset\Omega$ be such that (i) $A\subseteq F$ and (ii) that $T \not\subseteq F$. Let $\overline{\Lambda}_{T}^l$ 
be defined as in Lemma~\ref{th:th2_v}. 
If $\Lambda_F^u- \Lambda_T^l$ has negative eigenvalues, then for any element $f$ in the 
span of the eigenfunctions of the negative eigenvalues, \textcolor{black}{there exist a constant $\lambda^0_f>0$ such that}
\begin{equation}
\label{eqn:testcon}
    \textcolor{black}{\langle\overline{\Lambda}_{A}(\lambda f)-\overline{\Lambda}_T(\lambda f),\lambda f\rangle<0 \quad \forall\lambda<\lambda^0_f}.
\end{equation}
\end{prop}
\begin{proof}
\textcolor{black}{Let $\varphi_k$ be the eigenfunctions of $\Lambda_F^u- \Lambda_T^l$ related to the negative eigenvalues $\delta_k$ and let us consider an element $f$ in the span of $\{\varphi_k\}_k$, i.e.}
\begin{equation*}
\textcolor{black}{
    f(x)=\sum_k \beta_k \varphi_k(x).}
\end{equation*}
\textcolor{black}{By the linearity of $\Lambda_F^u- \Lambda_T^l$, it follows that}
\begin{equation}\label{eqn:eig_th}
\textcolor{black}{
    \langle \overline{\Lambda}_F^u f -\overline{\Lambda}_T^l f,f\rangle=\frac{1}{2} \sum_k \delta_k \beta_k^2  \lVert \varphi_k \rVert^2=:c_0<0,}
\end{equation}
\textcolor{black}{where $\lVert \varphi_k \rVert^2=\langle \varphi_k,\varphi_k \rangle$.}

\textcolor{black}{From Lemma \ref{th:th2_v}, we know that }    
\begin{equation}
\label{eqn:lim}
\textcolor{black}{
    \lim_{\lambda\to 0^+}\frac{\langle\overline{\Lambda}_T(\lambda f),\lambda f\rangle}{\lambda^2}  \geq\langle \overline{\Lambda}_{T}^l f,f\rangle,}
\end{equation}
\textcolor{black}{while from~\eqref{eqn:firstIneq2}}
\begin{equation*}
\textcolor{black}{
    \frac{\langle\overline{\Lambda}_A(\lambda f),\lambda f\rangle}{\lambda^2}  \leq \frac{\langle\overline{\Lambda}_F^u(\lambda f),\lambda f\rangle}{\lambda^2}  = \langle \overline{\Lambda}_{F}^u f,f\rangle.}
\end{equation*}

\textcolor{black}{
Let $\varepsilon=\alpha \lvert c_0 \rvert > 0$, with $0<\alpha<1$. From \eqref{eqn:lim} it turns out that there exists $\lambda^0_f > 0$ such that 
    \begin{equation}\label{eqn:liminf}
        \frac{\langle\overline{\Lambda}_T(\lambda f),\lambda f\rangle}{\lambda^2}  \geq\langle \overline{\Lambda}_{T}^l f,f\rangle - \varepsilon \quad \forall\, \lambda < \lambda_f^0.
    \end{equation}
Therefore, it follows that}
\begin{equation}\label{eqn:inprop}
\textcolor{black}{
    \frac{\langle\overline{\Lambda}_{A}(\lambda f)-\overline{\Lambda}_T(\lambda f),\lambda f\rangle}{\lambda^2} \leq \frac{\langle\overline{\Lambda}_{F}^u(\lambda f)-\overline{\Lambda}_T^l(\lambda f),\lambda f\rangle}{\lambda^2}+\varepsilon.
    }
\end{equation}
\textcolor{black}{Combining~\eqref{eqn:inprop} and~\eqref{eqn:eig_th}}
\begin{equation*}
\textcolor{black}{
    \frac{\langle\overline{\Lambda}_{A}(\lambda f)-\overline{\Lambda}_T(\lambda f),\lambda f\rangle}{\lambda^2} \leq \frac{\langle\overline{\Lambda}_{F}^u(\lambda f)-\overline{\Lambda}_T^l(\lambda f),\lambda f\rangle}{\lambda^2} + \varepsilon=c_0+\alpha \lvert c_0 \rvert <0.}
\end{equation*}
\end{proof}

\begin{rem}
\textcolor{black}{
    The proper boundary potential to be applied in order to reveal if $T \not\subseteq A$ is not $f$ but, rather, its scaled version $\lambda f$ as appears from \eqref{eqn:testcon}.}
\end{rem}

\begin{rem}
    \textcolor{black}{
    Proposition \ref{th:th3} and Lemma \ref{th:th2_v} provide an algorithmic method to evaluate a proper value for the scaling factor $\lambda$, to be applied in the front of $f$ in \eqref{eqn:testcon}.
    Specifically, it is first necessary to compute the quantity $c_0$ as defined in \eqref{eqn:eig_th}, for a specific boundary data $f$ of the span of the eigenvectors of $\Lambda_F^u- \Lambda_T^l$ with negative eigenvalues. Then, $\varepsilon$ is set to be equal to $\alpha \lvert c_0 \rvert$, with $0<\alpha<1$ and, starting from an initial guess, the scaling factor $\lambda$ is decreased until \eqref{eqn:liminf} is satisfied.}

    \textcolor{black}{
    Since, the material properties $\gamma_T^l$ and $\gamma_T$ are known, the calculations in \eqref{eqn:liminf} can be carried out before the measurement process, with no impact on the real-time capability of the method.} 
\end{rem}

\begin{rem}
\textcolor{black}{
    The method proposed can be applied without any modification also to the case of multiple intersections between the nonlinear characteristic $\gamma_{nl}$ and the upper bound $c_{bg}^u$ to the linear characteristic $\gamma_{bg}$. Indeed, in this case the value $s_0$ represents the minimum abscissa for which $\gamma_{nl}(s)=c_{bg}^u$.}
\end{rem}

\subsection{Choice of the fictitious anomalies}
\label{sec:fian}
This Section addresses the problem of designing the fictitious anomalies $F$. Their design is related to (i) the shape of a test anomaly $T$ and (ii) to the kind of defect details it is required to reconstruct. Although the fictitious anomalies can be chosen arbitrarily, we suggest a systemic approach to achieve excellent performance levels with a minimum computational cost. 

Specifically, for rectangular test anomalies, the tangent fictitious anomalies represented in Figure~\ref{fig_06_anf} can be utilized. 
\begin{figure}[htp]
\centering
\subfloat[][]
{\includegraphics[width=.3\textwidth]{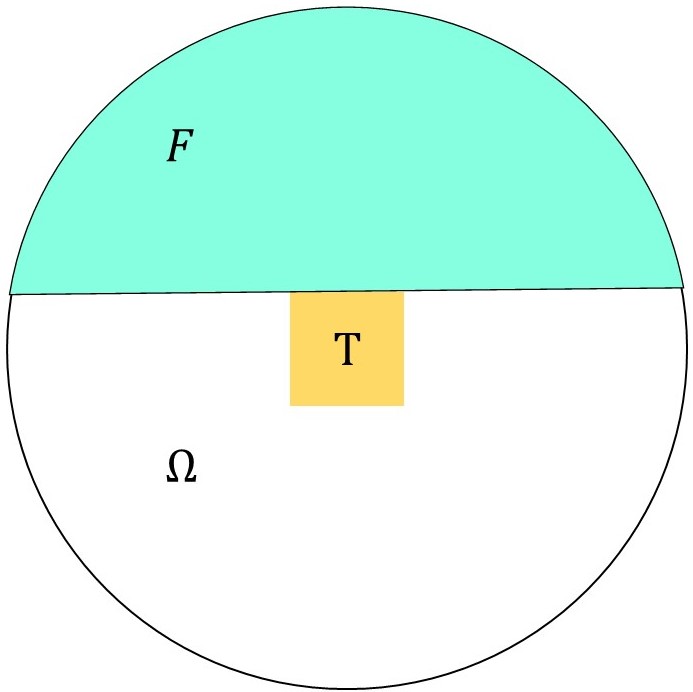}} \quad
\subfloat[][]
{\includegraphics[width=.3\textwidth]{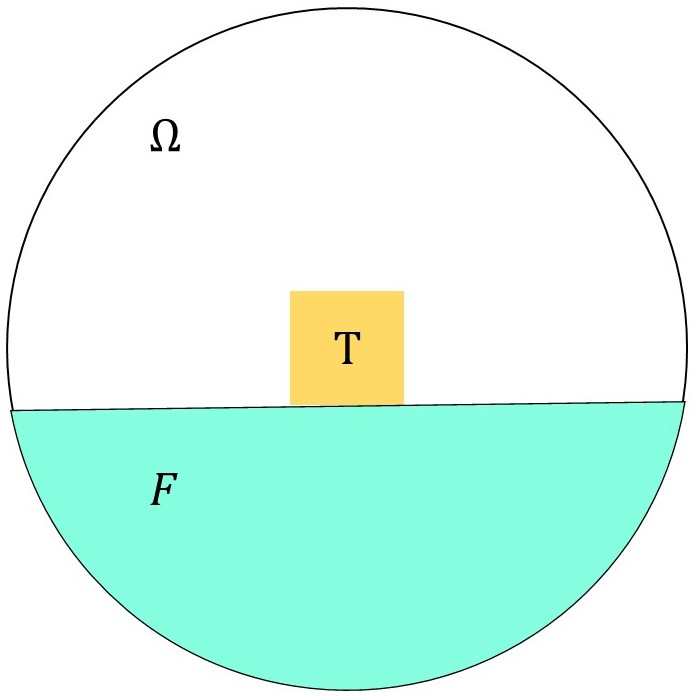}} \\
\subfloat[][]
{\includegraphics[width=.3\textwidth]{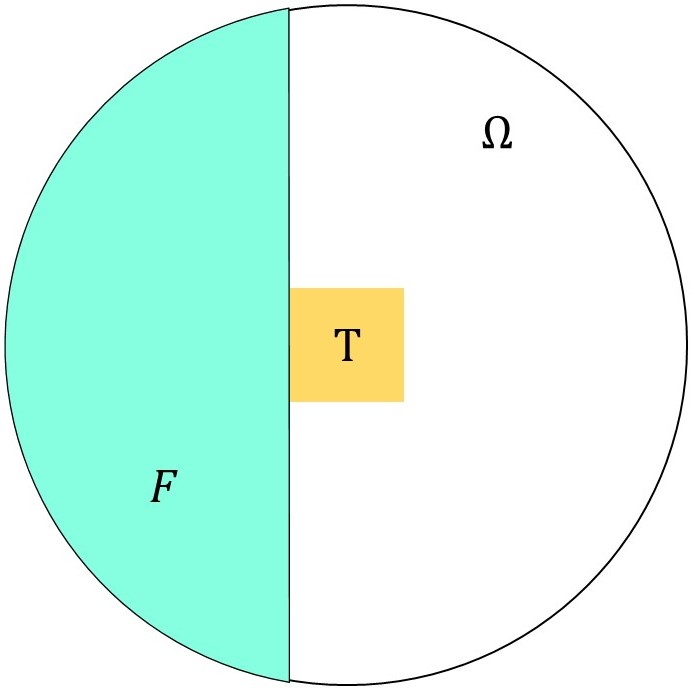}} \quad
\subfloat[][]
{\includegraphics[width=.3\textwidth]{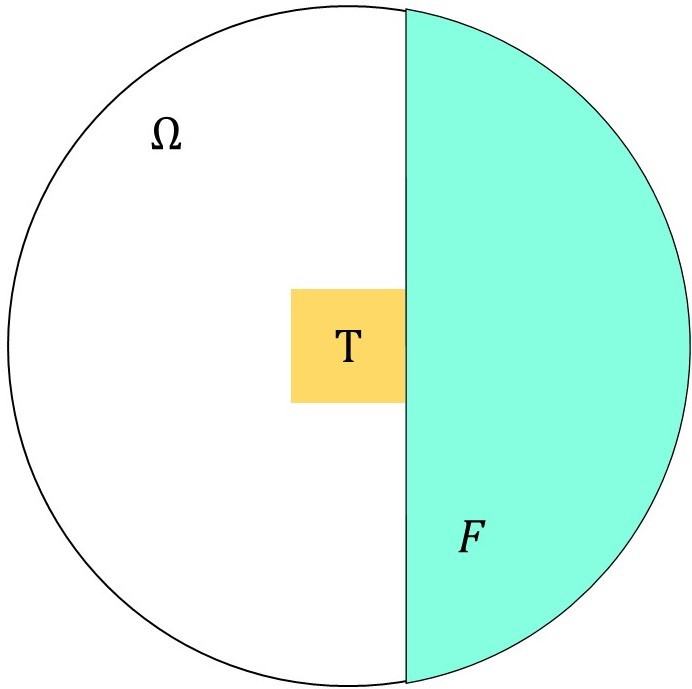}}
\caption{Example of fictitious anomalies $F$ utilized to determine the boundary potentials.}
\label{fig_06_anf}
\end{figure}
This idea can be extended to general shaped test anomalies such as, for instance, circles. In this case, a tangent line on the boundary of $T$ is selected. The tangent line divides the domain $\Omega$ into two parts. The fictitious anomaly is represented by the part that does not contain $T$. Therefore, for each test anomaly $T$, a certain number of boundary points are chosen, and for each point a fictitious anomaly can be evaluated via the tangent line, as in Figure~\ref{fig_07_fict2}.

\begin{figure}[htp]
\centering\includegraphics[width=0.325\textwidth]{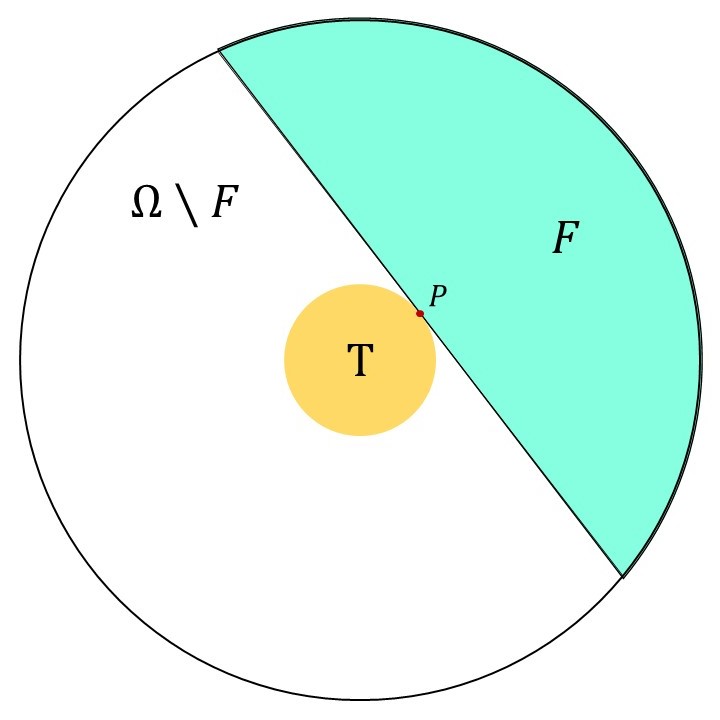}
    \caption{A fictitious anomaly $F$ for a convex test anomaly $T$: the case of a circle.}
    \label{fig_07_fict2}
\end{figure}

To reconstruct concave defects, such as L-shaped and C-shaped defects, it is necessary to introduce some other fictitious anomalies. Figure~\ref{fig_08_fict3} shows possible fictitious anomalies related to a squared test anomaly, and aimed at the treatment of concave unknown anomalies $A$.
\begin{figure}[htp]
\centering
\subfloat[][]
{\includegraphics[width=.3\textwidth]{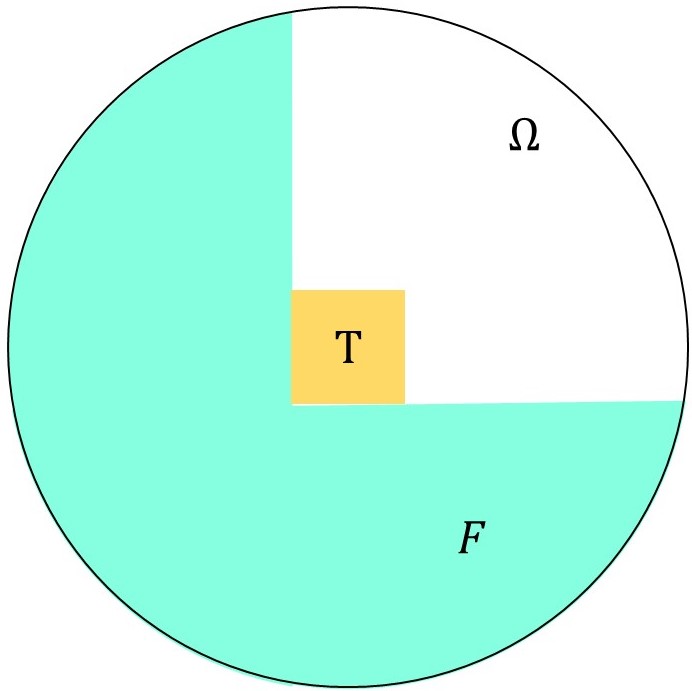}} \quad
\subfloat[][]
{\includegraphics[width=.3\textwidth]{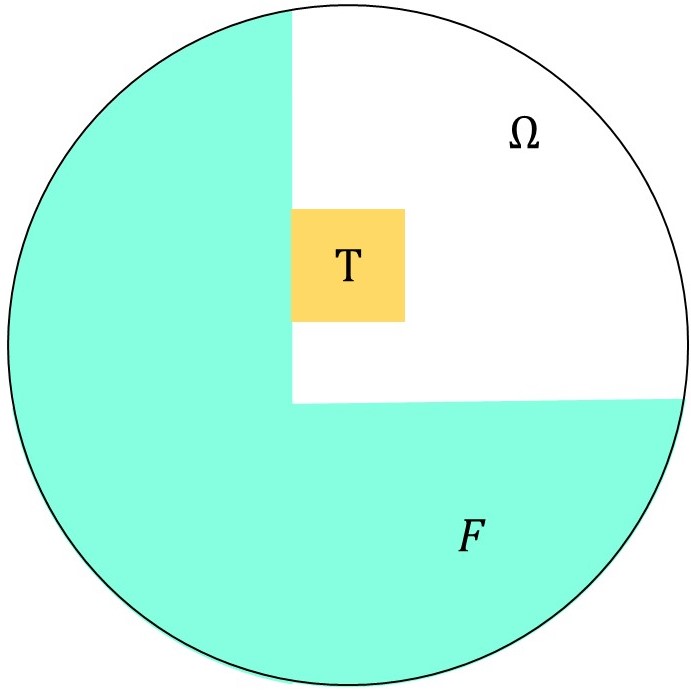}} \\
\subfloat[][]
{\includegraphics[width=.3\textwidth]{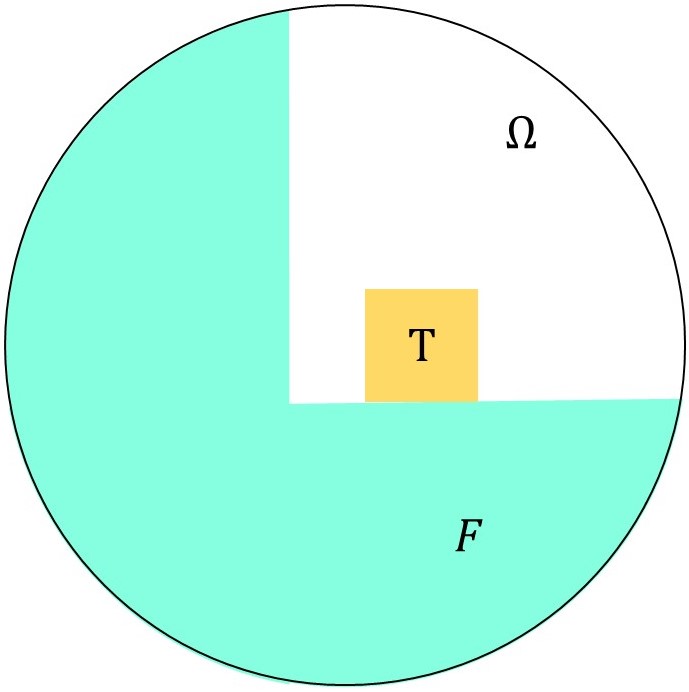}} 
\caption{Example of anomalies $F$ for concave defects.}
\label{fig_08_fict3}
\end{figure}
The extension for general (convex) shaped test anomalies can be achieved by introducing two tangent lines, as shown in Figure \ref{fig_09_fict4}. Each tangent line divides the domain $\Omega$ into two parts. The fictitious anomaly is the union of the two regions that do not contain $T$.
\begin{figure}[htp]
    \centering
    \includegraphics[width=0.325\textwidth]{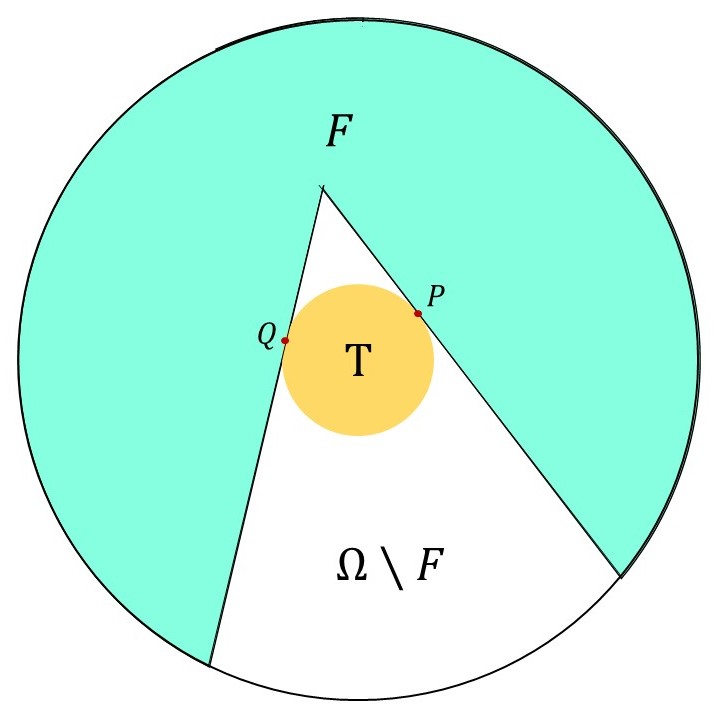}
    \caption{Fictitious anomaly $A_i$ for an arbitrarily shaped test anomaly $T$ and concave defects}
    \label{fig_09_fict4}
\end{figure}


\subsection{Limit of the proposed method}\label{sec:lim}
In this Section the limits of the proposed method are highlighted. The foundation of the method is its ability to reveal when a test anomaly $T$ is \emph{not} contained in $A$.

In order to reveal whether $T \nsubseteq A$, a key aspect is the choice of $f$, the prescribed (test) boundary value. In Sections \ref{sec:basid} and \ref{sec:advid}, $f$ was found under the assumption that the unknown anomaly $A$ is contained in the fictitious anomaly $F$, and that the test anomaly $T$ is not contained in $F$.

However, there are  situations where $T \nsubseteq A$ and the latter condition $\left( A \subseteq F \right)$ is not satisfied, as shown in Figure~\ref{fig_10_lim}.
\begin{figure}[htp]
   \centering
   \subfloat[][]
   {\includegraphics[width=.3\textwidth]{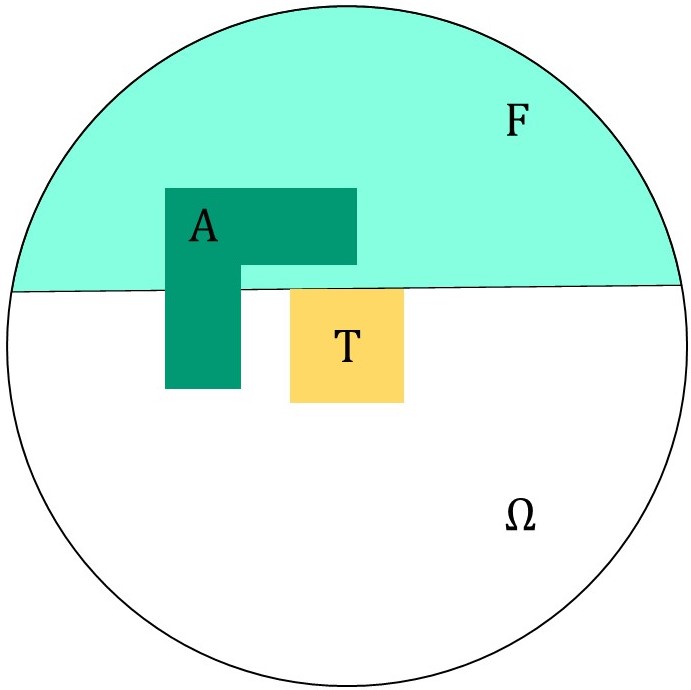}} \quad
   \subfloat[][]
   {\includegraphics[width=.3\textwidth]{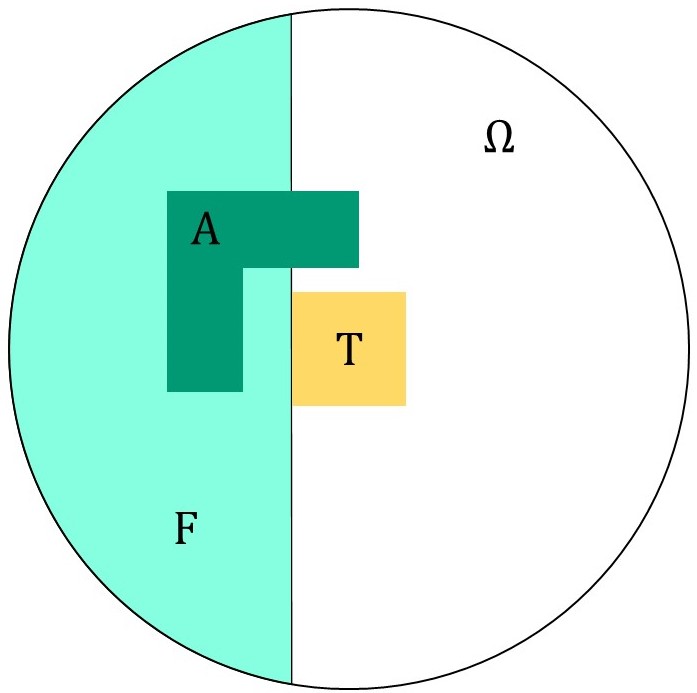}} \\
   \caption{The actual $A$ is not completely included in any fictitious anomalies $F$.}
   \label{fig_10_lim}
\end{figure}
In these cases $\left( A \not\subseteq F \right)$, it is not possible to guarantee that the boundary potentials arising from the methods of Sections \ref{sec:basid} and \ref{sec:advid}, are able to reveal that $T\not\subseteq A$. 

However, this limit can be easily mitigated because: (i) we only need $A$ to be contained in at least one of the sets of type $F$, when $A \nsubseteq T$ and (ii) we may design large domains $F$ to reduce the risk of $A \not\subseteq F$ (see Figure~\ref{fig_08_fict3}). A detailed analysis of this aspect is beyond the scope of this work.

 
\section{Numerical Examples}\label{sec5}
In this section we report some numerical examples of reconstructions in $\mathbb{R}^2$, \textcolor{black}{for the magnetostatic and the steady currents cases. The domain $\Omega$ is a circle, with radius equal to $r=30\,\text{cm}$ in the magnetostatic case and with radius equal to $r=3\,\text{cm}$ in the steady currents case.} All the reconstructions are obtained by adding synthetic noise to the simulated data generated via a FEM code. The constitutive relationship are those of Figures \ref{fig_11_sigmae} and \ref{fig_12_perm2}.

\begin{figure}[htb]
    \centering
    \includegraphics[width=1\textwidth]{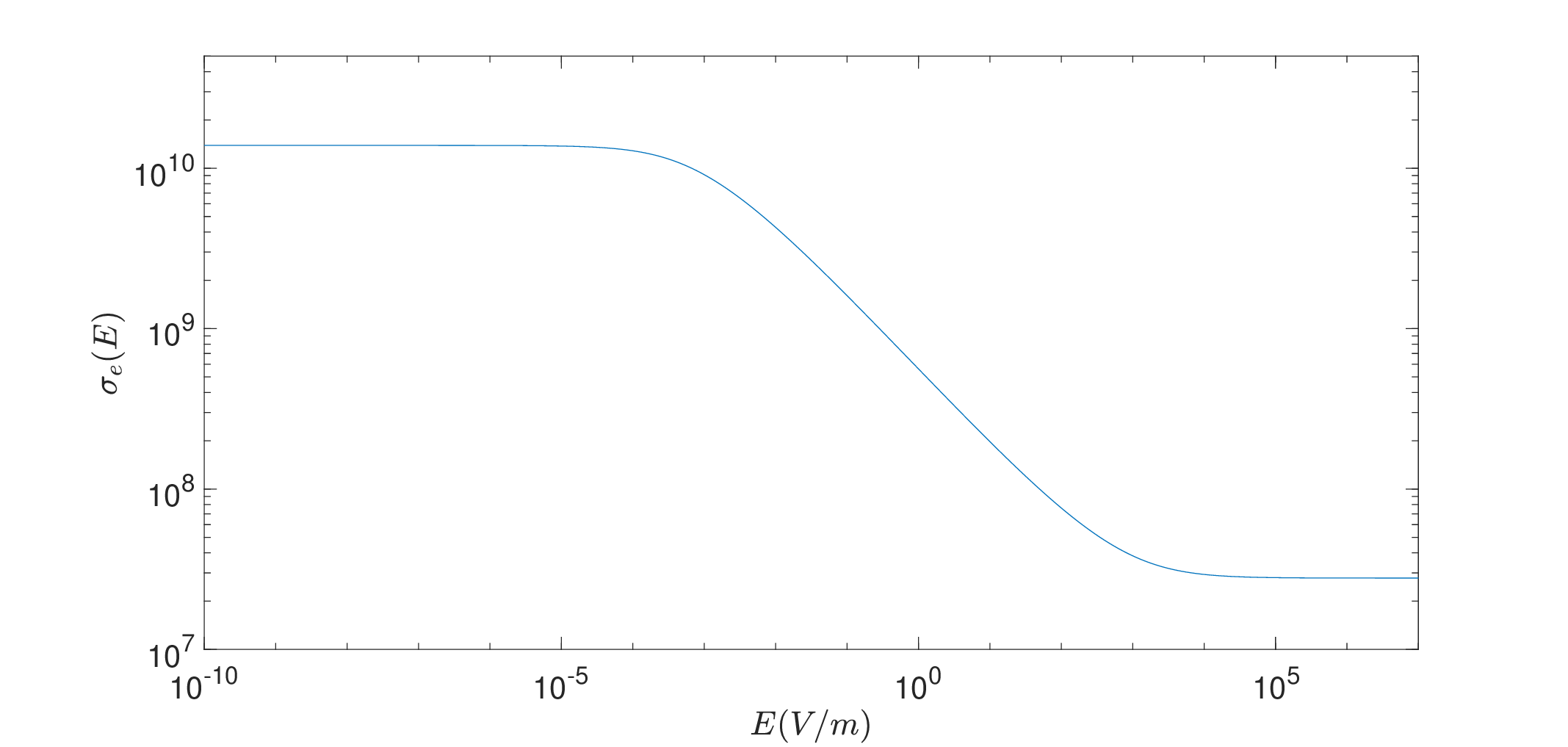}
    \caption{Effective electrical conductivity $\sigma_e$. The lower bound is $\sigma_l=2.7861\times 10^7\,\text{S/m}$ and the upper bound is $\sigma_u=1.3875\times 10^{10}\,\text{S/m}$.}
    \label{fig_11_sigmae}
\end{figure}
\begin{figure}[htp]
    \centering
    \includegraphics[width=1\textwidth]{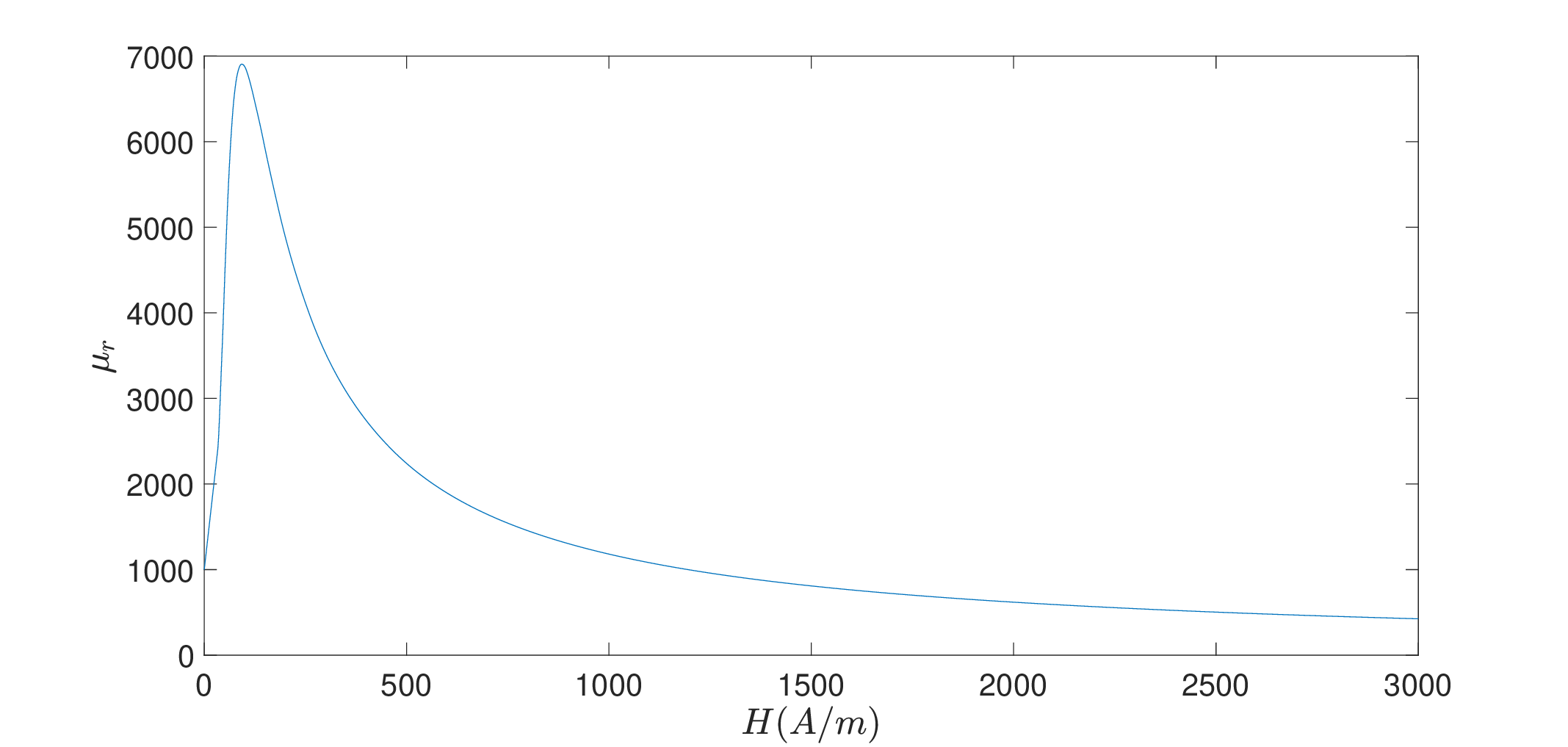}.
    \caption{Nonlinear (relative) magnetic permeability of $M330\mhyphen50A$ Electrical Steel by Surahammars Bruk AB~\cite{magperm}.}
    \label{fig_12_perm2}
\end{figure}

In order to deal with noisy measurements, the reconstruction rule introduced in~\eqref{eqn:r0} has to be slightly modified, in line with \cite{art:Ha15, art:Ta16_1}.

\subsection{Treatment of the noise}
\textcolor{black}{The measurement system measures $\langle\overline{\Lambda}_{A}(f),f\rangle$, which corresponds to the Dirichlet energy for a prescribed boundary data $f$. From the physical standpoint, the Dirichlet energy is either the magnetic co-energy (magnetostatic case), or the electric co-energy (electrostatic case) or the average Ohmic power (steady currents case). 
}

\textcolor{black}{A transducer converts the measured quantity $\langle\overline{\Lambda}_{A}(f),f\rangle$ into a voltage $\OM_A(f)$:
\begin{equation}
\label{eqn:trns}
    \OM_A(f)=k \langle\overline{\Lambda}_{A}(f),f\rangle,
\end{equation}
where $k$ is the transducer constant. Then, $\OM_A(f)$
is measured by means of a multimeter. To be specific on the noise levels, it is assumed that the multimeter is the 2002 8\textonehalf-Digit High Performance Multimeter by Keithley~\cite{web:in}. The noise model is, therefore, given by}

\begin{equation*}
    \textcolor{black}{\OMtilde_A(f)=\OM_A(f)(1+\eta_1\xi_1)+\eta_2\xi_2L}, 
\end{equation*}
\textcolor{black}{where $\OMtilde_A(f)$ is the noisy version of $\OM_A(f)$, $\xi_1$ and $\xi_2$ are two random variables uniformly distributed in the interval $[-1,1]$, $L$ is the measurement range and $\eta_1$ and $\eta_2$ are two parameters giving the noise level amplitude (see \cite{web:in}). The noise consists of two distinct terms: one controlled by $\eta_1$, proportional to the measured quantity, and another controlled by $\eta_2$, proportional to the selected measurement range.}

\textcolor{black}{
We have the following Proposition.
\begin{prop}
\label{prop:noisetreat}
    Given $0 \le  \eta_1 < 1$, $\eta_2 \ge 0$ and $L \ge 0$, $\left|\xi_1\right| \le 1$ and $\left|\xi_2\right| \le 1$, we have
    \begin{equation}
    \label{eqn:MPNoisy}
    T \subseteq A \implies \OM_T(f) \le \frac{\OMtilde_A(f)+\eta_2 L}{1-\eta_1}, \, \forall f \in X_{\diamond}(\partial\Omega).
    \end{equation}
\end{prop}
\begin{proof}
     Since $\left|\xi_1\right| \le 1$ and $\left|\xi_2\right| \le 1$,
\begin{equation}\label{eqn:rn}
     \OM_A(f) \leq \frac{\OMtilde_A(f)+\eta_2 L}{1-\eta_1}.
\end{equation}
Finally, equation \eqref{eqn:MPNoisy} follows since $T \subseteq A$ implies $\OM_T(f) \le \OM_A(f), \,\forall f \in X_{\diamond}(\partial\Omega)$, because of the Monotonicity Principle.
\end{proof}}
Consequently, the reconstruction rule of \eqref{eqn:r0} is changed as follows
\begin{equation}\label{eqn:rn0}
\textcolor{black}{
        \widetilde{A}_D^U=\bigcup_i \left\{T_i \:\bigg |\: \frac{\OMtilde_A(f_{i,j,k})+\eta_2 L}{1-\eta_1} - \OM_T(f_{i,j,k}) \geq 0 \quad\forall\,j,k \right\}}.
\end{equation}
Proposition \ref{prop:noisetreat} and the reconstruction rule of \eqref{eqn:rn0}, guarantee that a test anomaly $T_i$ contained in the actually anomaly $A$ is not discarded in the presence of noise, in line with \cite{art:Ha15, art:Ta16_1}. Moreover, this proves that if $A$ can be represented by the union of test anomalies, then $A \subseteq \widetilde{A}_D^U$ when noise is present.

\subsection{Numerical Model}
\textcolor{black}{The numerical simulations have been carried out with an in-house Finite Element Method. The finite element mesh discretization of $\Omega$, for both the magnetostatic and steady currents cases, consists of $N_t=16512$ triangular first-order elements and $N_p=8385$ nodes. $\partial\Omega$ is discretized in $N_b=256$ elements. The scalar potential $u$ is discretized as $u(x) = \sum_{i=1}^{N_{i}} x_i \varphi_i(x)$, where $N_{i}=N_p-N_{ext}$ is the number of internal nodes, the $x_i$s are the Degrees of Freedom and the $\varphi_i$s are first-order nodal shape functions.}

\textcolor{black}{The numerical model for solving \eqref{eqn:dirint}, based on the Galerkin method, is}
\begin{equation}\label{eqn:dirintnum}
    \textcolor{black}{\sum_{i=1}^{N_{int}} x_i\int_{\Omega}\gamma\left(x,\left\lvert \sum_{i=1}^{N_{i}} x_i \nabla\varphi_i(x) \right\rvert\right)\nabla \varphi_i(x)\cdot \nabla\varphi_j(x)\,dx=0\quad j=1,\hdots, N_{i}.}
\end{equation}
\textcolor{black}{The system of nonlinear equations \eqref{eqn:dirintnum} is numerically solved by means of the Newton-Raphson method.}

\textcolor{black}{The boundary data $f$ is numerically approximated by a piecewise linear function. Since it has zero average, it can been represented by $N_{b}-1$ DoFs. In this discrete setting, the DtN and the average DtN, act from $\mathbb{R}^{N_{b}-1}$ to $\mathbb{R}^{N_{b}-1}$.}

\subsection{Numerical results}
\textcolor{black}{In this section, the proposed imaging method is validated via numerical simulations. Different shapes of different type of the unknown anomaly $A$ are considered, in two different physical settings: steady currents and magnetostatic. Figures~\ref{fig_13_rec} and~\ref{fig_14_rec2} show the reconstructions, obtained for a total of $2195$ boundary potentials applied for each configuration. The estimated anomaly $\widetilde{A}_D^U$ is in white, while the boundary of the unknown region $A$ is marked in red. The values for $\eta_1$ and $\eta_2$, representing the noise level parameters, are those for the 2002 8\textonehalf-Digit High Performance Multimeter by Keithley. The transducer constant $k$ of \eqref{eqn:trns} is chosen in order to provide an output voltage in the range from $200\,\text{mV}$ to $20\,\text{V}$. Specifically, $k=1\times 10^{-2}\,\text{V$\cdot$m/W}$ and $k=7\times 10^6\,\text{V$\cdot$ m/J}$ for the steady currents and magnetostatic cases, respectively. In this range, the multimeter guarantees $\eta_1$ and $\eta_2$ of the order of $10^{-6}$ (see Table~\ref{tab_02_eta} for details).}

\begin{figure}[htp]
    \centering
    \subfloat[][\emph{Circle}]
    {\includegraphics[width=.3\textwidth]{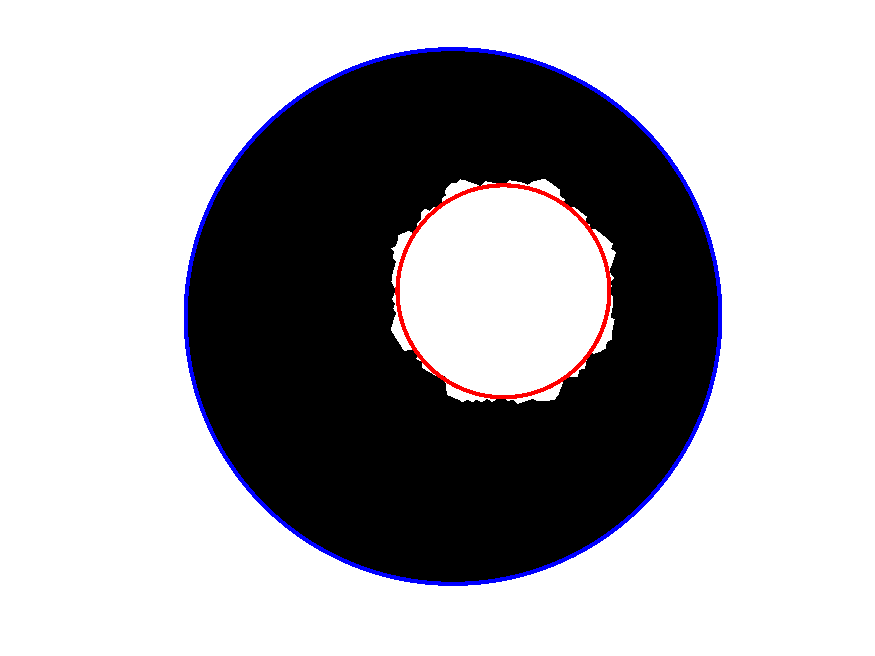}} \quad
    \subfloat[][\emph{Ellipse}]
    {\includegraphics[width=.3\textwidth]{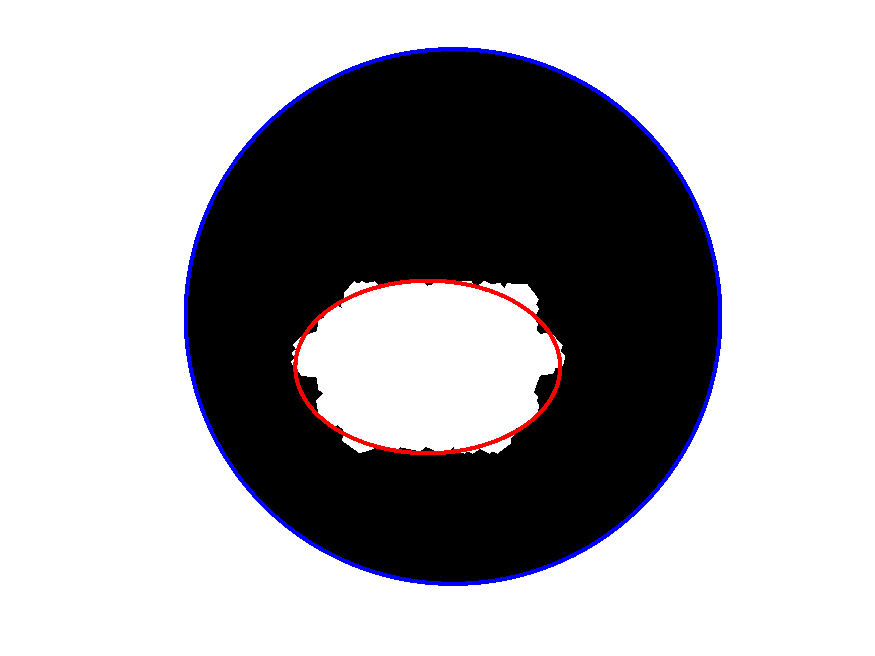}} \quad
    \subfloat[][\emph{Peanut}]
    {\includegraphics[width=.3\textwidth]{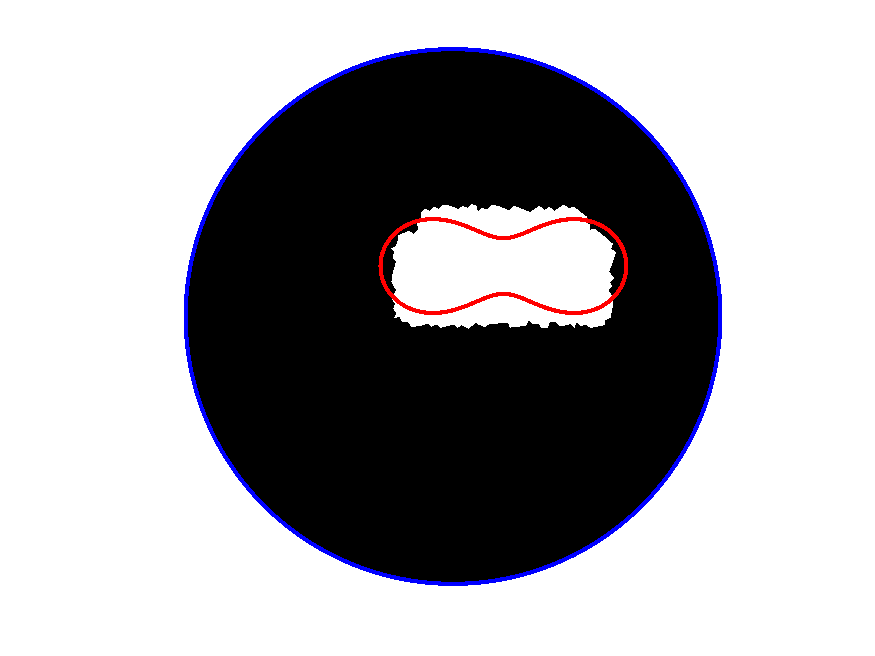}} \\
    \subfloat[][\emph{Droplet}]
    {\includegraphics[width=.3\textwidth]{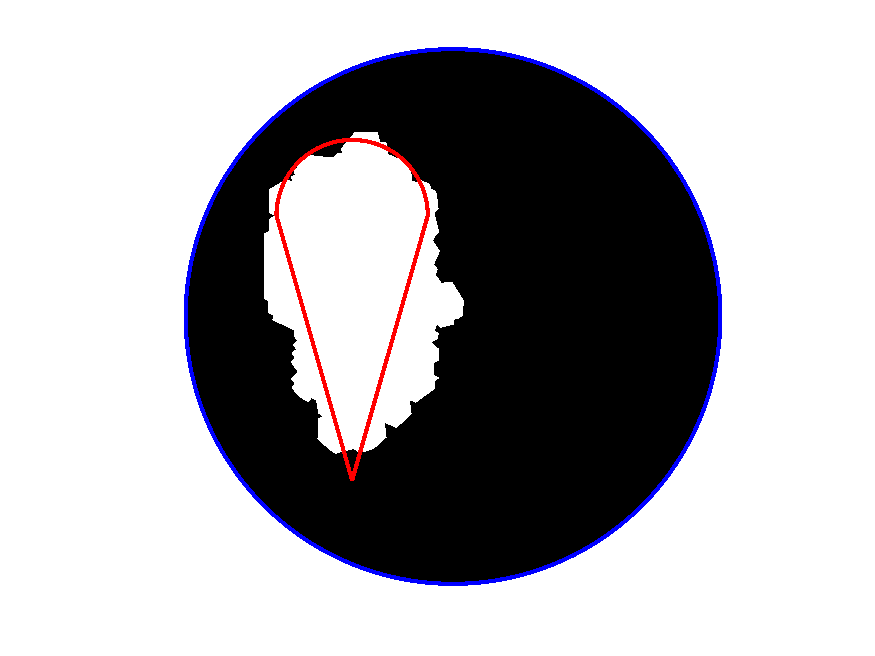}} \quad
    \subfloat[][\emph{Kite}]
    {\includegraphics[width=.3\textwidth]{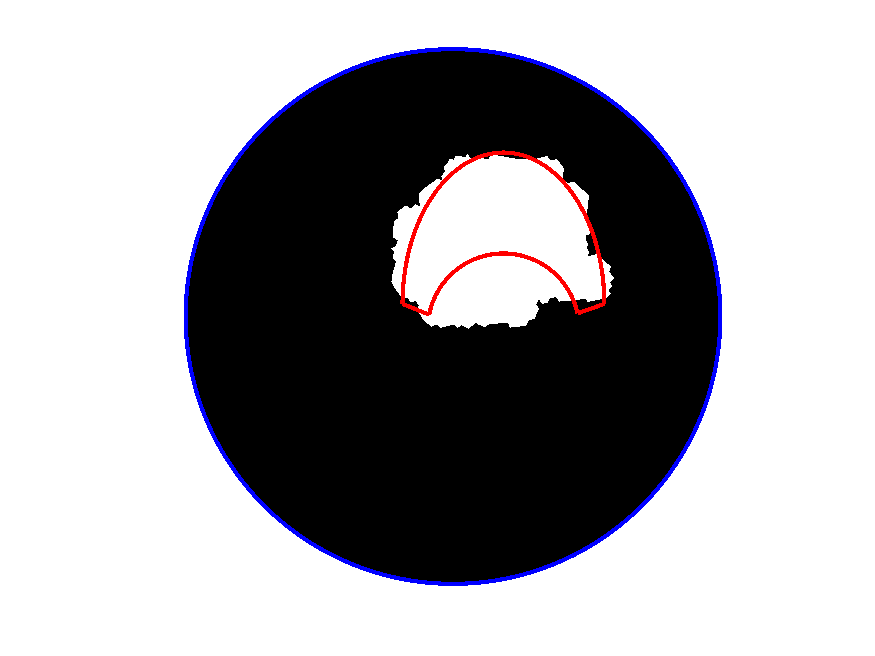}} \quad
    \subfloat[][\emph{Test anomaly 1}]
    {\includegraphics[width=.3\textwidth]{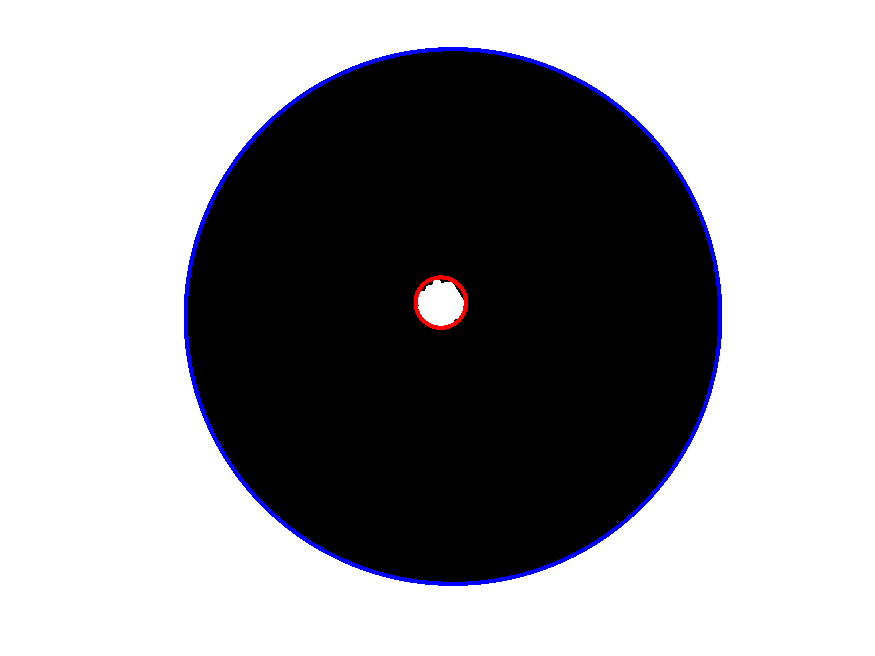}} \\
    \subfloat[][\emph{Test anomaly 2}]
    {\includegraphics[width=.3\textwidth]{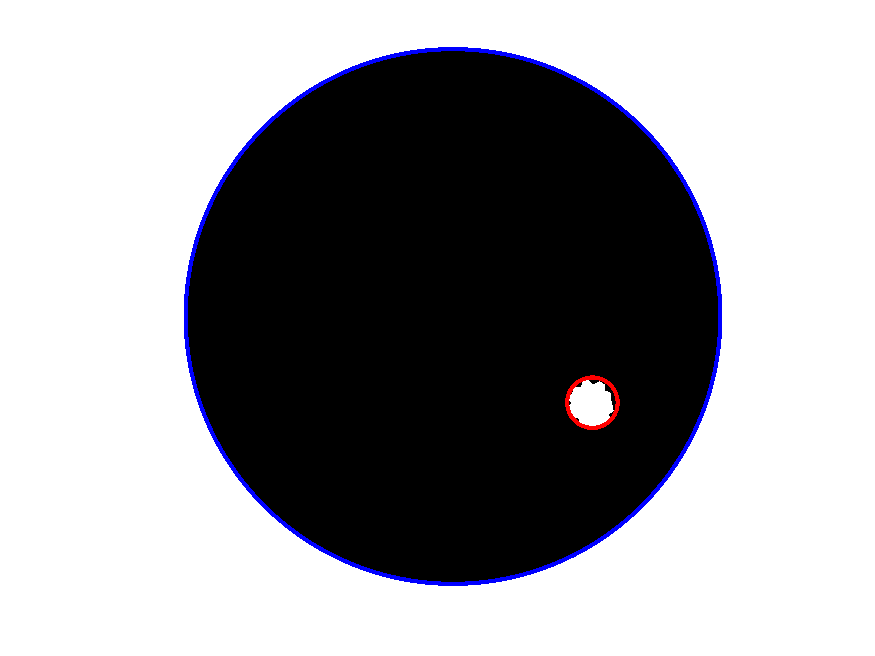}} \quad
    \subfloat[][\emph{Two connected components}]
    {\includegraphics[width=.3\textwidth]{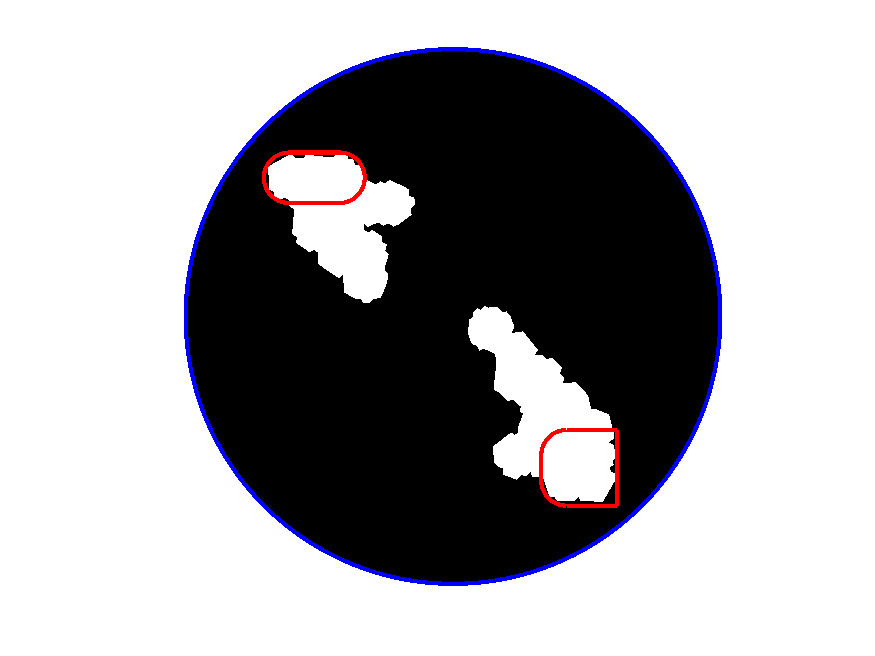}} \quad
    \subfloat[][\emph{Hollow Circle}]
    {\includegraphics[width=.3\textwidth]{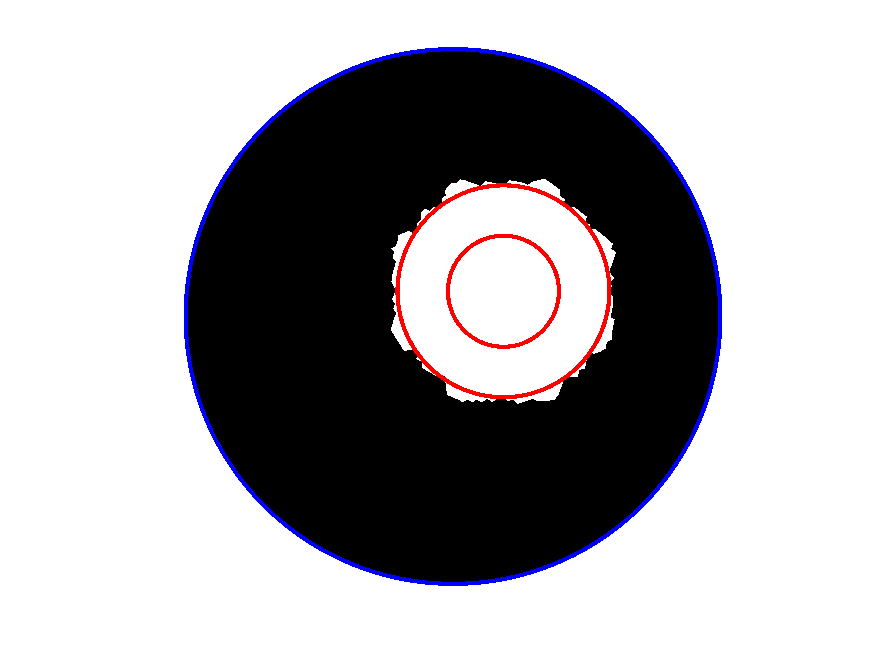}}
    \caption{Reconstructions carried out by the proposed method. The estimated anomaly $\tilde{A}_D^U$ is shown in white while the boundary of the anomaly $A$ is marked in red. \textcolor{black}{The nonlinear phase is given by the mixture of superconductive spherical inclusions in a linear medium. The nonlinear material is surrounded by steel.}}
    \label{fig_13_rec}
\end{figure}

\begin{figure}[htp]
    \centering
    \subfloat[][\emph{Circle}]
    {\includegraphics[width=.3\textwidth]{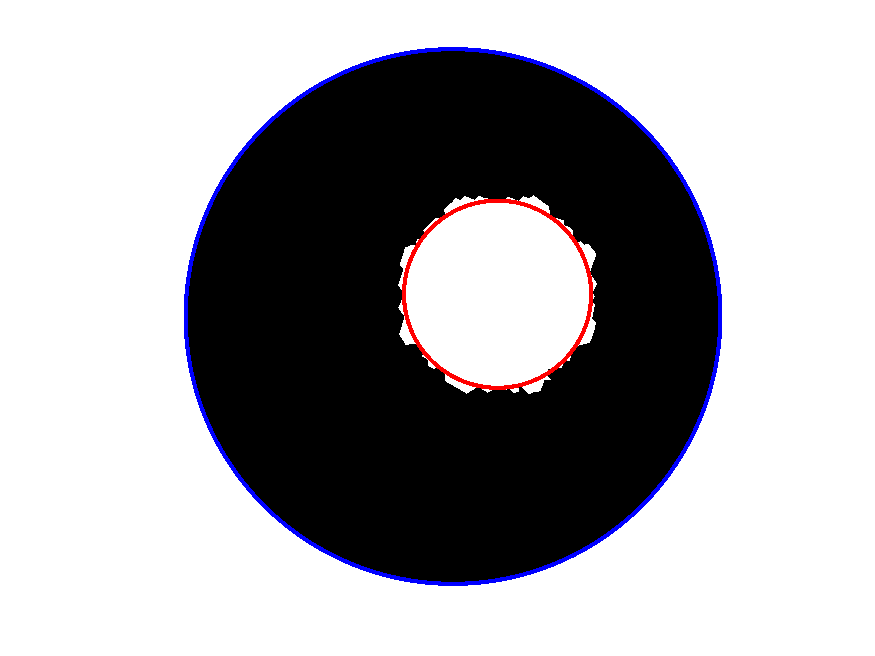}} \quad
    \subfloat[][\emph{Ellipse}]
    {\includegraphics[width=.3\textwidth]{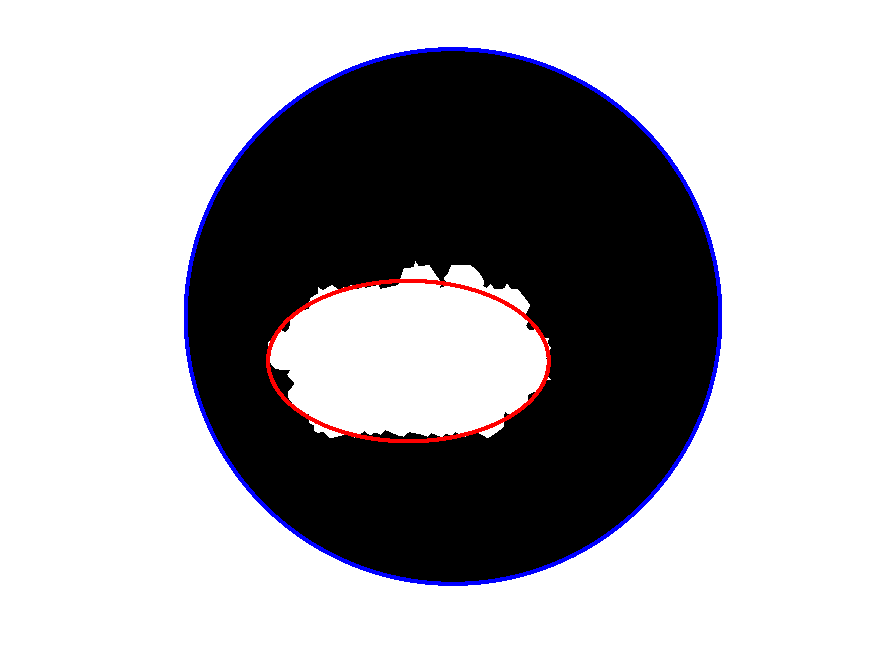}} \quad
    \subfloat[][\emph{Peanut}]
    {\includegraphics[width=.3\textwidth]{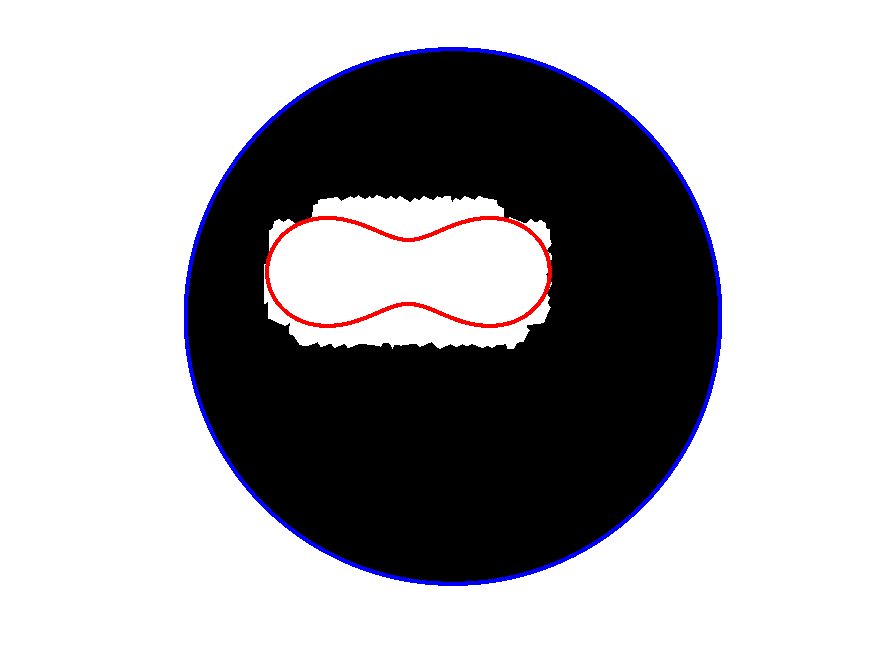}} \\
    \subfloat[][\emph{Droplet}]
    {\includegraphics[width=.3\textwidth]{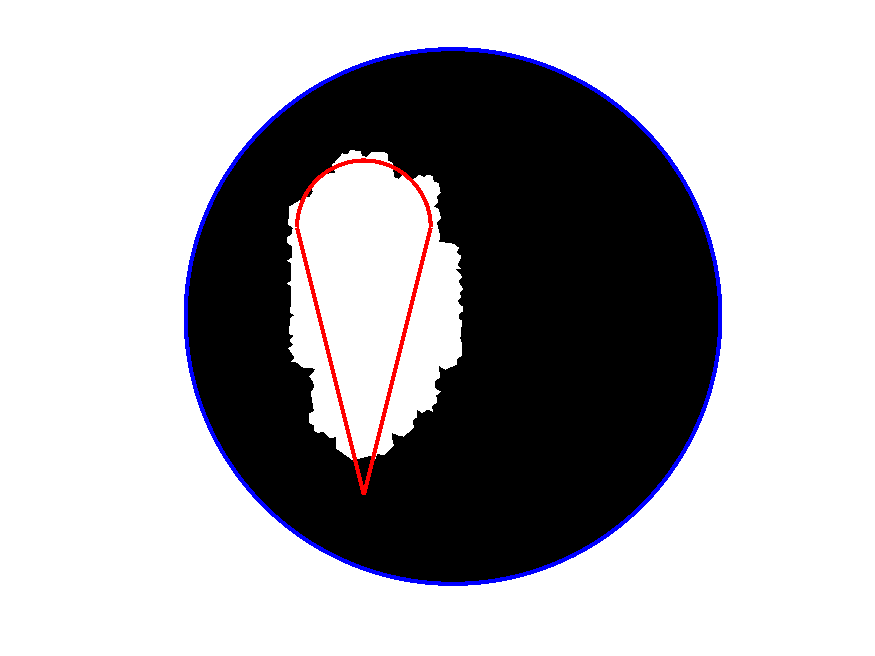}} \quad
    \subfloat[][\emph{Kite}]
    {\includegraphics[width=.3\textwidth]{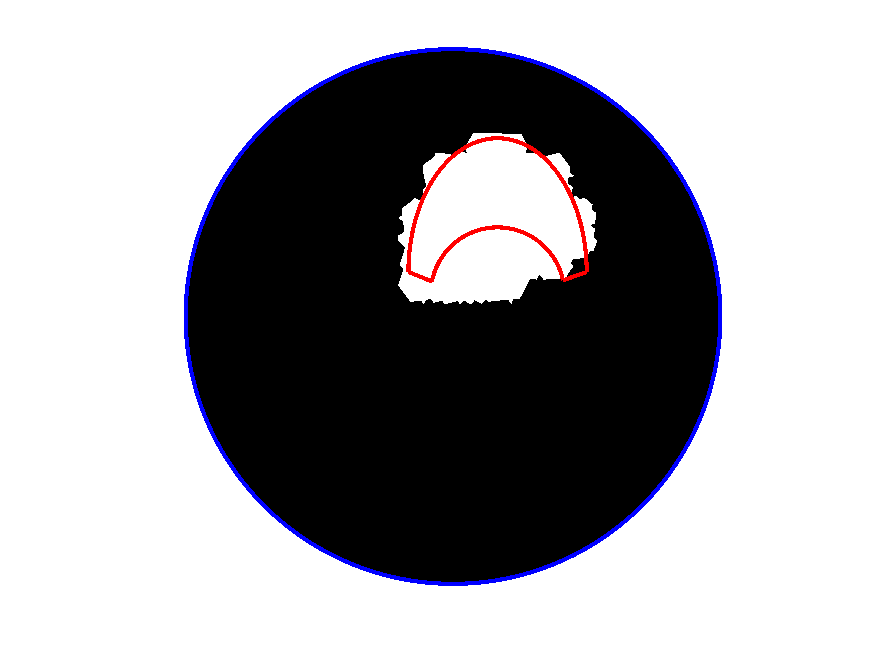}} \quad
    \subfloat[][\emph{Test anomaly 1}]
    {\includegraphics[width=.3\textwidth]{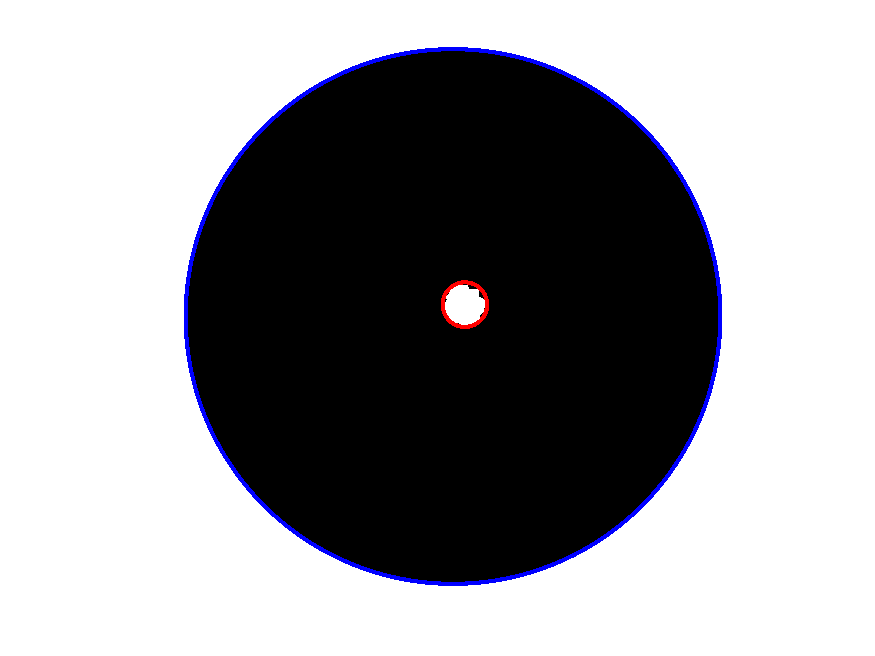}} \\
    \subfloat[][\emph{Test anomaly 2}]
    {\includegraphics[width=.3\textwidth]{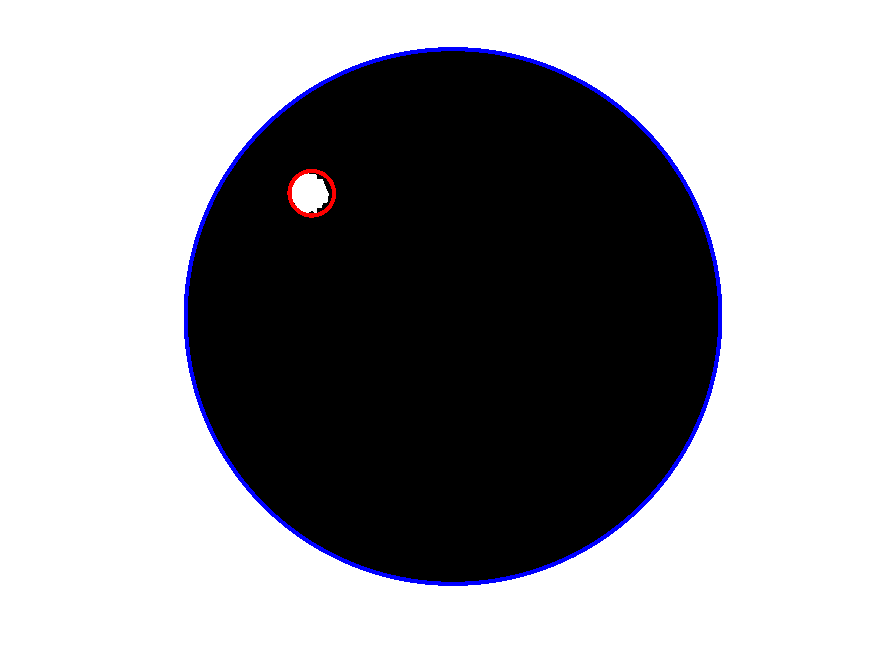}} \quad
    \subfloat[][\emph{Two connected components}]
    {\includegraphics[width=.3\textwidth]{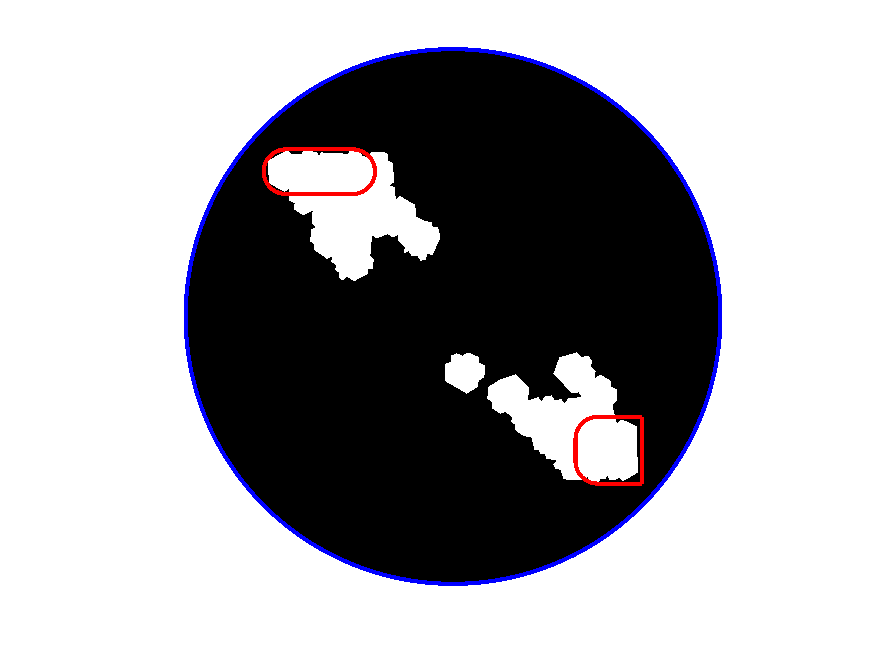}} \quad
    \subfloat[][\emph{Hollow Circle}]
    {\includegraphics[width=.3\textwidth]{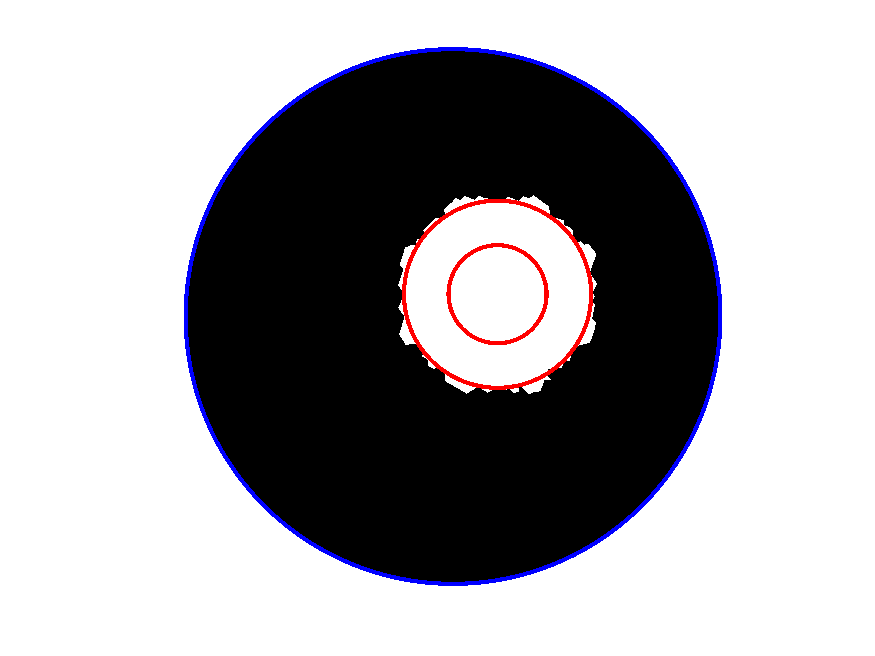}}
    \caption{Reconstructions carried out by the proposed method. The estimated anomaly $\tilde{A}_D^U$ is shown in white while the boundary of the anomaly $A$ is marked in red. \textcolor{black}{The nonlinear phase is given by $M330\mhyphen50A$ Electrical Steel, while $\mu_{bg}=\mu_0$.}}
    \label{fig_14_rec2}
\end{figure}

\begin{table}[htb]
    \centering
    \begin{tabular}{ccc}
        \toprule
        Range & $\eta_1$ & $\eta_2$ \\
        \midrule
        $200\,\text{mV}$ & $3.5 \times 10^{-6}$ & $3.0 \times 10^{-6}$ \\
        $2\,\text{V}$ & $1.2 \times 10^{-6}$ & $0.3 \times 10^{-6}$ \\
        $20\,\text{V}$ & $1.2 \times 10^{-6}$ & $0.1 \times 10^{-6}$ \\
        \bottomrule
    \end{tabular}
    \caption{Values for the noise levels $\eta_1$ and $\eta_2$ adopted in the simulations. The values in this table are those reported in the datasheet of the 2002 8\textonehalf-Digit High Performance Multimeter by Keithley~\cite{web:in}.}
    \label{tab_02_eta}
\end{table}

\textcolor{black}{For the sake of completeness, Figure~\ref{fig_15_eda} shows the distribution of the Dirichlet Energy $\langle \overline{\Lambda}_A(f),f\rangle$ for all the applied test potentials, with reference to the unknown anomaly of Figure~\ref{fig_13_rec}(A) and~\ref{fig_14_rec2}(A). The test potentials, evaluated as in Section~\ref{sec4}, produce value of Dirichlet Energy in a reasonable range, from the experimental point of view. 
}

\begin{figure}[htp]
    \centering
    \subfloat[][\emph{Steady currents case}]
    {\includegraphics[width=.8\textwidth]{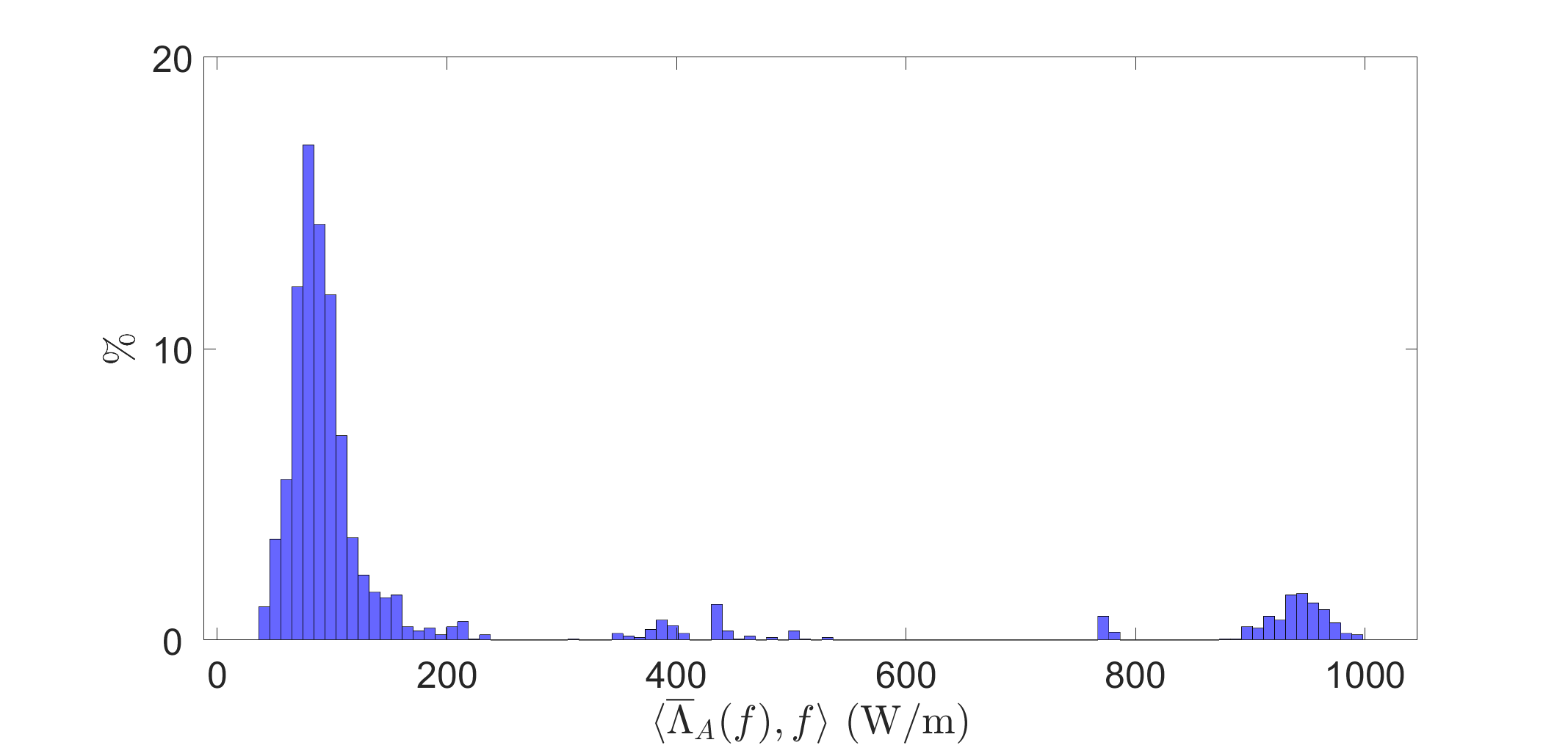}} \\
    \subfloat[][\emph{Magnetostatic case}]
    {\includegraphics[width=.8\textwidth]{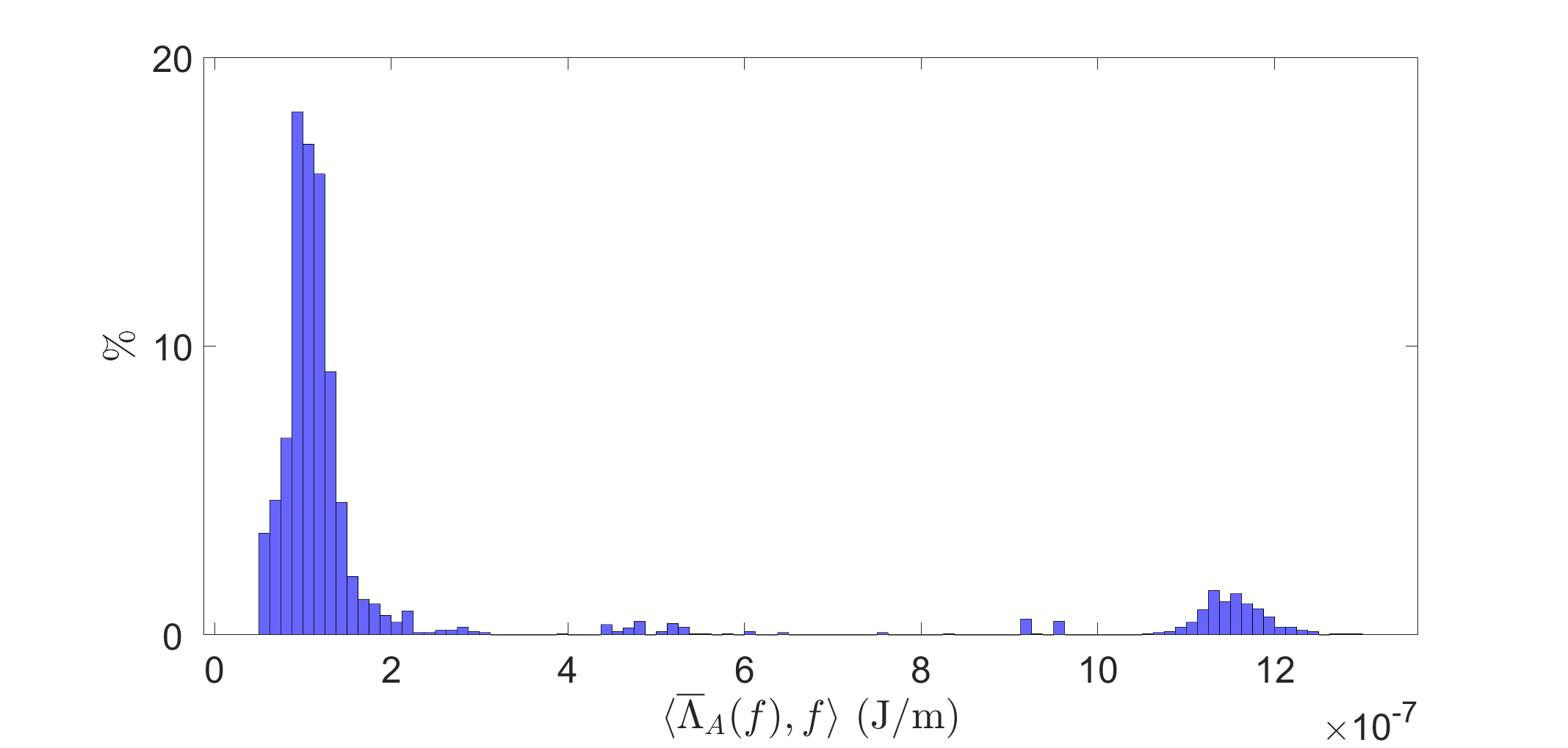}} \\
    \caption{Distribution of the Dirichlet Energy resulting from the application of the test boundary potentials for the configurations of Figure~\ref{fig_13_rec}(A) and ~\ref{fig_14_rec2}(A). A total of 2195 boundary potentials have been applied.}
    \label{fig_15_eda}
\end{figure}

\subsubsection{Steady currents case}
\textcolor{black}{The first series of reconstructions (see Figure~\ref{fig_13_rec}) concerns the steady currents problem. Specifically, we consider a background made of steel, with electrical conductivity $\sigma_{bg}=10\times 10^6\,\text{S/m}$, and a nonlinear phase made of a composite material (see Figure \ref{fig_11_sigmae}).}

\textcolor{black}{The composite material is made of a mixture of superconductive spherical inclusions in a linear medium. The superconducting material is that used for standard commercial superconductive wires. The superconducting spheres are embedded in a linear matrix made of a Ag-Mg alloy, with electrical conductivity equal to $\sigma_1=55.5\times 10^6\,\text{S/m}$ \cite{art:AgMg}. This matrix is typically used as stabilizer in Bi-2212 superconductive wires. The constitutive relationship of the circular superconducting inclusions is $E(J)=E_0\left({J}/{J_c}\right)^n$ (\emph{E-J Power Law}, see \cite{rhyner1993magnetic})},
\textcolor{black}{which leads to the nonlinear electrical conductivity}
\begin{equation*}
    \textcolor{black}{
        \sigma_2(E)=\frac{J_c}{E_0}\left(\frac{E}{E_0}\right)^{\frac{1-n}{n}},
    }
\end{equation*}
\textcolor{black}{where $E_0$ is equal to $1\times 10^{-4}\,\text{V/m}$, $J_c=8000\,\text{A/mm$^2$}$ and $n=27$ for the material of Bi-2212 superconductive wires \cite{art:bi2212}.}

\textcolor{black}{The effective electrical conductivity of the mixture $\sigma_e(\cdot)$ is evaluated via the Bruggeman formula \cite{art:brug} as a solution of}
\begin{equation}\label{eqn:brug}
    \textcolor{black}{
        \delta _{1}\,{\frac{\sigma _{1}-\sigma _{e}(E)}{\sigma _{1}+2\sigma _{e}(E)}}+\delta _{2}\,{\frac{\sigma _{2}(E)-\sigma _{e}(E)}{\sigma _{2}(E)+2\sigma _{e}(E)}}\,=\,0,
    }
\end{equation}
\textcolor{black}{where $\delta_i$ is the volume fraction of the $i-$th component. For this specific mixture $\delta_1=0.668$ and $\delta_2=0.332$. The resulting effective electrical conductivity is shown in Figure~\ref{fig_11_sigmae}.}

\subsubsection{Magnetostatic case}
\textcolor{black}{The second series of reconstructions (see Figure~\ref{fig_14_rec2}) is inspired by a security application. Indeed, as} pointed out in Section~\ref{sec1}, an interesting field for Magnetostatic Permeability Tomography applications is the detection of magnetic materials in boxes and containers for surveillance and security reasons. Driven by this motivation, the following shows some examples whose aim is to reconstruct the shape, position and dimension of an anomaly made by a ferromagnetic material in air. In particular, the anomaly is characterized by the nonlinear magnetic permeability reported in Figure~\ref{fig_12_perm2} and $\mu_{bg}=\mu_0$.

\subsubsection{Discussion of the numerical results}
\textcolor{black}{Figure~\ref{fig_13_rec} and~\ref{fig_14_rec2} show similar performances of the proposed method in the two different physical scenarios. Furthermore, although the overall quality of the reconstructions is good, it is possible to highlight some peculiar behaviours: excellent performances on convex anomalies (A,B,F,G), tendency to convexification (C,D,E), possibility to resolve multiple inclusions (H), impossibility to correctly reconstruct anomalies with cavities (I).}

\section{Conclusions}\label{sec13}
This work is focused on the problem of tomography governed by elliptic PDEs in the presence of nonlinear materials. This is an inverse problem where the aim is to find the nonlinear and spatially varying coefficient of an elliptic PDE.

Based on the recent development of the Monotonocity Principle for nonlinear elliptical PDEs, this contribution proposes a tomographic method for the inverse obstacle problem of nonlinear anomalies embedded in a linear background. This is the first time that a method for the tomography of nonlinear materials is presented.

One of the main challenges in dealing with nonlinear problems, lies in selecting a proper boundary potential capable of revealing the presence of anomalies, if any.
A major result (see Section \ref{sec4}) is that such boundary data can be found by solving a linear problem, despite the nonlinear nature of the original problem. Moreover, these boundary data can be pre-computed and stored once for all. This key result is fundamental in making the imaging method suitable for real-time tomography, which is a rare feature in the field of inverse problems.

Finally, numerical examples in the framework of Electrical Resistance Tomography and Magnetic Permeability Tomography have demonstrated the excellent performance achieved by the method.

\bibliographystyle
{ieeetr}
\bibliography{biblioCFPPT}

\begin{thebibliography}{10}

\bibitem{calderon1980inverse}
A.~Calder\'on, ``On an inverse boundary,'' in {\em Seminar on Numerical Analysis and its Applications to Continuum Physics, Rio de Janeiro, 1980}, pp.~65--73, Brazilian Math. Soc., 1980.

\bibitem{art:Mar15}
L.~Marmugi, S.~Hussain, C.~Deans, and F.~Renzoni, ``Magnetic induction imaging with optical atomic magnetometers: towards applications to screening and surveillance,'' in {\em SPIE Security + Defence}, 2015.

\bibitem{book:Dorn18}
O.~Dorn and A.~Hiles, {\em A level set method for magnetic induction thomography of 3D boxes and containers}.
\newblock Netherlands: IOS Press, 2018.

\bibitem{art:So05}
M.~Soleimani and W.~Lionheart, ``Image reconstruction in three-dimensional magnetostatic permeability tomography,'' {\em IEEE Transactions on Magnetics}, vol.~41, no.~4, pp.~1274--1279, 2005.

\bibitem{art:Ig03}
H.~Igarashi, K.~Ooi, and T.~Honma, ``A magnetostatic reconstruction of permeability distribution in material,'' in {\em Inverse Problems in Engineering Mechanics IV} (M.~Tanaka, ed.), pp.~383--388, Amsterdam: Elsevier Science B.V., 2003.

\bibitem{art:Super19}
B.~C. Robert, M.~U. Fareed, and H.~S. Ruiz, ``How to choose the superconducting material law for the modelling of 2g-hts coils,'' {\em Materials}, vol.~12, p.~2679, 8 2019.

\bibitem{art:Bu08}
P.~R. Bueno, J.~A. Varela, and E.~Longo, ``Sno2, zno and related polycrystalline compound semiconductors: An overview and review on the voltage-dependent resistance (non-ohmic) feature,'' {\em Journal of the European Ceramic Society}, vol.~28, no.~3, pp.~505--529, 2008.

\bibitem{art:Me18}
R.~Metz, S.~Boucher, M.~Hassanzadeh, and A.~Solaiappan, ``Interest of nonlinear zno/silicone composite materials in cable termination,'' {\em Material Science \& Engineering International Journal}, vol.~2, 06 2018.

\bibitem{art:tisnl}
S.~\v{C}orovi\'c, I.~Lackovic, P.~\v{S}u\v{s}tari\v{c}, T.~Sustar, T.~Rodic, and D.~Miklavcic, ``Modeling of electric field distribution in tissues during electroporation,'' {\em Biomedical engineering online}, vol.~12, p.~16, 02 2013.

\bibitem{art:skin_nl}
D.~Panescu, J.~Webster, and R.~Stratbucker, ``A nonlinear electrical-thermal model of the skin,'' {\em IEEE Transactions on Biomedical Engineering}, vol.~41, no.~7, pp.~672--680, 1994.

\bibitem{art:die_nl}
L.~Padurariu, L.~Curecheriu, V.~Buscaglia, and L.~Mitoseriu, ``Field-dependent permittivity in nanostructured batio${}_{3}$ ceramics: Modeling and experimental verification,'' {\em Phys. Rev. B}, vol.~85, p.~224111, 6 2012.

\bibitem{art:diode_nl}
T.~Yamamoto, S.~Suzuki, H.~Suzuki, K.~Kawaguchi, K.~T.~K. Takahashi, and Y.~Y.~Y. Yoshisato, ``Effect of the field dependent permittivity and interfacial layer on ba1-xkxbio3/nb-doped srtio3 schottky junctions,'' {\em Japanese Journal of Applied Physics}, vol.~36, p.~L390, 4 1997.

\bibitem{art:Lam20}
K.~F. Lam and I.~Yousept, ``Consistency of a phase field regularisation for an inverse problem governed by a quasilinear maxwell system,'' {\em Inverse Problems}, vol.~36, p.~045011, 3 2020.

\bibitem{art:Sa12}
M.~Salo and X.~Zhong, ``An inverse problem for the {$p$-Laplacian}: Boundary determination,'' {\em SIAM Journal on Mathematical Analysis}, vol.~44, no.~4, pp.~2474--2495, 2012.

\bibitem{art:Bra14}
T.~Brander, ``Calder\'on problem for the {$p$-Laplacian}: First order derivative of conductivity on the boundary,'' {\em Proceedings of the American Mathematical Society}, vol.~144, 03 2014.

\bibitem{art:Ca20}
C.~I. C\^{a}rstea and M.~Kar, ``Recovery of coefficients for a weighted p-laplacian perturbed by a linear second order term,'' {\em Inverse Problems}, vol.~37, p.~015013, 12 2020.

\bibitem{art:Bra15}
T.~Brander, M.~Kar, and M.~Salo, ``Enclosure method for the {$p$-Laplace} equation,'' {\em Inverse Problems}, vol.~31, p.~045001, 2 2015.

\bibitem{art:Salo16}
C.-Y. Guo, M.~Kar, and M.~Salo, ``Inverse problems for $p$-laplace type equations under monotonicity assumptions,'' {\em Rend. Istit. Mat. Univ. Trieste}, vol.~48, 05 2016.

\bibitem{art:Co21}
A.~{Corbo Esposito}, L.~Faella, G.~Piscitelli, R.~Prakash, and A.~Tamburrino, ``Monotonicity principle in tomography of nonlinear conducting materials,'' {\em Inverse Problems}, vol.~37, p.~045012, 4 2021.

\bibitem{MPMETHODS}
A.~{Corbo Esposito}, L.~Faella, V.~Mottola, G.~Piscitelli, R.~Prakash, and A.~Tamburrino, ``Monotonicity principle for the imaging of piecewise nonlinear materials,'' {\em Preprint}, 2023.

\bibitem{art:Co23}
A.~{Corbo Esposito}, L.~Faella, G.~Piscitelli, R.~Prakash, and A.~Tamburrino, ``Monotonicity principle for tomography in nonlinear conducting materials,'' {\em Journal of Physics: Conference Series}, vol.~2444, p.~012004, 2 2023.

\bibitem{art:Ta02}
A.~Tamburrino and G.~Rubinacci, ``A new non-iterative inversion method for electrical resistance tomography,'' {\em Inverse Problems}, vol.~18, pp.~1809--1829, 11 2002.

\bibitem{art:Ta06}
A.~Tamburrino and G.~Rubinacci, ``Fast methods for quantitative eddy-current tomography of conductive materials,'' {\em IEEE Transactions on Magnetics}, vol.~42, no.~8, pp.~2017--2028, 2006.

\bibitem{art:Ta06p}
A.~Tamburrino, ``Monotonicity based imaging methods for elliptic and parabolic inverse problems,'' {\em Journal of Inverse and Ill-posed Problems}, vol.~14, no.~6, pp.~633--642, 2006.

\bibitem{art:Ca12}
F.~Calvano, G.~Rubinacci, and A.~Tamburrino, ``Fast methods for shape reconstruction in electrical resistance tomography,'' {\em NDT \& E International}, vol.~46, pp.~32--40, 2012.

\bibitem{art:Ta03}
A.~Tamburrino, G.~Rubinacci, M.~Soleimani, and W.~Lionheart, ``Non iterative inversion method for electrical resistance, capacitance and inductance tomography for two phase materials,'' {\em 3rd World Congress on Industrial Process Tomography}, 01 2003.

\bibitem{art:Ta10}
A.~Tamburrino, S.~Ventre, and G.~Rubinacci, ``Recent developments of a monotonicity imaging method for magnetic induction tomography in the small skin-depth regime,'' {\em Inverse Problems}, vol.~26, p.~074016, 7 2010.

\bibitem{art:Su17}
Z.~Su, S.~Ventre, L.~Udpa, and A.~Tamburrino, ``Monotonicity based imaging method for time-domain eddy current problems*,'' {\em Inverse Problems}, vol.~33, p.~125007, 11 2017.

\bibitem{art:Ta16}
A.~{Tamburrino}, Z.~{Su}, S.~{Ventre}, L.~{Udpa}, and S.~{Udpa}, ``Monotonicity based imaging method in time domain eddy current testing,'' in {\em Electromagnetic Nondestructive Evaluation (XIX)}, vol.~41, pp.~1--8, March 2016.

\bibitem{art:Su17_2}
Z.~Su, L.~Udpa, G.~Giovinco, S.~Ventre, and A.~Tamburrino, ``Monotonicity principle in pulsed eddy current testing and its application to defect sizing,'' in {\em 2017 International Applied Computational Electromagnetics Society Symposium - Italy (ACES)}, pp.~1--2, 2017.

\bibitem{book:Ta15}
A.~Tamburrino, L.~Barbato, D.~Colton, and P.~Monk, ``Imaging of dielectric objects via monotonicity of the transmission eigenvalues,'' in {\em Proc. of The 12th International Conference on Mathematical and Numerical Aspects of Wave Propagation}, pp.~20--24, 2015.

\bibitem{art:To20}
T.~Daimon, T.~Furuya, and R.~Saiin, ``The monotonicity method for the inverse crack scattering problem,'' {\em Inverse Problems in Science and Engineering}, vol.~28, no.~11, pp.~1570--1581, 2020.

\bibitem{garde2022reconstruction}
H.~Garde and N.~Hyv{\"o}nen, ``Reconstruction of singular and degenerate inclusions in calder{\'o}n's problem,'' {\em Inverse Problems and Imaging}, vol.~16, no.~5, pp.~1219--1227, 2022.

\bibitem{albicker2023monotonicity}
A.~Albicker and R.~Griesmaier, ``Monotonicity in inverse scattering for maxwell's equations,'' {\em Inverse Problems and Imaging}, vol.~17, no.~1, pp.~68--105, 2023.

\bibitem{albicker2020monotonicity}
A.~Albicker and R.~Griesmaier, ``Monotonicity in inverse obstacle scattering on unbounded domains,'' {\em Inverse Problems}, 2020.

\bibitem{kar2023fractional}
M.~Kar, J.~Railo, and P.~Zimmermann, ``The fractional p-biharmonic systems: optimal poincar{\'e} constants, unique continuation and inverse problems,'' {\em Calculus of Variations and Partial Differential Equations}, vol.~62, no.~4, p.~130, 2023.

\bibitem{tamburrino2021themonotonicity}
A.~Tamburrino, G.~Piscitelli, and Z.~Zhou, ``The monotonicity principle for magnetic induction tomography,'' {\em Inverse Problems}, vol.~37, no.~9, p.~095003, 2021.

\bibitem{art:Ha15}
B.~Harrach and M.~Ullrich, ``Resolution guarantees in electrical impedance tomography,'' {\em IEEE Transactions on Medical Imaging}, vol.~34, no.~7, pp.~1513--1521, 2015.

\bibitem{art:Ta16_1}
A.~Tamburrino, A.~Vento, S.~Ventre, and A.~Maffucci, ``Monotonicity imaging method for flaw detection in aeronautical applications,'' {\em Studies in Applied Electromagnetics and Mechanics}, vol.~41, pp.~284--292, 01 2016.

\bibitem{daimon2020monotonicity}
T.~Daimon, T.~Furuya, and R.~Saiin, ``The monotonicity method for the inverse crack scattering problem,'' {\em Inverse Problems in Science and Engineering}, pp.~1--12, 2020.

\bibitem{art:Ha13}
B.~Harrach and M.~Ullrich, ``Monotonicity-based shape reconstruction in electrical impedance tomography,'' {\em SIAM Journal on Mathematical Analysis}, vol.~45, no.~6, pp.~3382--3403, 2013.

\bibitem{art:Gi90}
D.~G. Gisser, D.~Isaacson, and J.~C. Newell, ``Electric current computed tomography and eigenvalues,'' {\em SIAM Journal on Applied Mathematics}, vol.~50, no.~6, pp.~1623--1634, 1990.

\bibitem{corboesposito2023thep0laplacesignature}
A.~Corbo~Esposito, L.~Faella, V.~Mottola, G.~Piscitelli, R.~Prakash, and A.~Tamburrino, ``The $p_0$-laplace signature for quasilinear inverse problems,'' {\em Siam J. Imaging Sci.}, in press.

\bibitem{corboesposito2023theplaplacesignature}
A.~Corbo~Esposito, L.~Faella, G.~Piscitelli, R.~Prakash, and A.~Tamburrino, ``The $p$-laplace signature for quasilinear inverse problems with large boundary data,'' {\em Siam J. Math. Anal.}, in press.

\bibitem{garde2022simplified}
H.~Garde, ``Simplified reconstruction of layered materials in eit,'' {\em Applied Mathematics Letters}, vol.~126, p.~107815, 2022.

\bibitem{arens2023monotonicity}
T.~Arens, R.~Griesmaier, and R.~Zhang, ``Monotonicity-based shape reconstruction for an inverse scattering problem in a waveguide,'' {\em Inverse Problems}, vol.~39, no.~7, p.~075009, 2023.

\bibitem{magperm}
J.~Pyrh\"onen and V.~H. Tapani~Jokinen, {\em Appendix A: Properties of Magnetic Sheets}, pp.~570--571.
\newblock John Wiley \& Sons, Ltd, 2013.

\bibitem{web:in}
``Available on-line at.'' https://www.tek.com/en/products/keithley/digital-multimeter/2002-series.

\bibitem{art:AgMg}
P.~Li, L.~Ye, J.~Jiang, and T.~Shen, ``Rrr and thermal conductivity of ag and ag-0.2 wt.\% mg alloy in ag/bi-2212 wires,'' {\em IOP Conference Series: Materials Science and Engineering}, vol.~102, p.~012027, 11 2015.

\bibitem{rhyner1993magnetic}
J.~Rhyner, ``Magnetic properties and ac-losses of superconductors with power law current-voltage characteristics,'' {\em Physica C: Superconductivity}, vol.~212, no.~3-4, pp.~292--300, 1993.

\bibitem{art:bi2212}
S.~Barua, D.~S. Davis, Y.~Oz, J.~Jiang, E.~E. Hellstrom, U.~P. Trociewitz, and D.~C. Larbalestier, ``Critical current distributions of recent bi-2212 round wires,'' {\em IEEE Transactions on Applied Superconductivity}, vol.~31, no.~5, pp.~1--6, 2021.

\bibitem{art:brug}
D.~A.~G. Bruggeman, ``Berechnung verschiedener physikalischer konstanten von heterogenen substanzen. i. dielektrizit\"{a}tskonstanten und leitf\"{a}higkeiten der mischk\"{o}rper aus isotropen substanzen,'' {\em Annalen der Physik}, vol.~416, no.~7, pp.~636--664, 1935.

\end{thebibliography}
\end{document}